\documentclass[11pt, preprint, reqno]{amsart}
\usepackage{graphicx} 

\usepackage{amsfonts, amsmath, amssymb}
\usepackage{amsthm}
\usepackage{amsrefs}
\bibstyle{alphabetic}

\usepackage{mathrsfs}
\usepackage{xcolor}

\usepackage{enumitem}

\usepackage[scr=esstix]{mathalfa}

\usepackage{hyperref}
\hypersetup{colorlinks={true},linkcolor={black},citecolor=black}
\usepackage[capitalize]{cleveref}

\renewcommand{\epsilon}{\varepsilon}
\renewcommand{\subset}{\subseteq}
\renewcommand{\supset}{\supseteq}
\newcommand{\dist}{\mathrm{dist}}
\newcommand{\img}{\mathrm{img}\, }
\newcommand{\diam}{\mathrm{diam}\, }
\newcommand{\supp}{\mathrm{supp}\, }

\theoremstyle{definition}
\newtheorem{definition}{Definition}[section]
\newtheorem{theorem}[definition]{Theorem}
\newtheorem{lemma}[definition]{Lemma}
\newtheorem{corollary}[definition]{Corollary}

\newtheorem{remark}[definition]{Remark}
\newtheorem{proposition}[definition]{Proposition}

\numberwithin{equation}{section}

 \thanks{The author's work was supported by an internal fellowship from the Massachusetts Institute of Technology School of Science. \copyright  \today}

\title[Structure of optimal free Dirichlet regions]{On the structure of optimal free Dirichlet regions in mass transportation problems}
\author{Lucas D. O'Brien}
\address{Lucas D. O'Brien \\ Department of Mathematics \\ Massachusetts Institute of Technology \\ Cambridge, Massachusetts \\ {Email: obrie720@mit.edu}}
\date{}

\begin{document}

\begin{abstract}
    For a compactly supported probability measure $\mu$ on the $d$-dimensional space $\mathbb{R}^d$, the average distance problem asks us to minimize the average distance functional over all compact, connected, $\Sigma \subseteq \mathbb{R}^d$ satisfying the Hausdorff $1$-measure constraint $\mathcal{H}^1(\Sigma) \leq \ell$. This problem was first introduced in 2002 by Buttazzo, Oudet, and Stepanov to study optimal transport problems with free regions on which the transport cost vanishes, and has undergone a considerable amount of research since. Most recently, Kobayashi, Kim, and the author studied the structure of these regions using the barycentre field, a tool for studying the average distance functional introduced previously by Kobayashi, Hayase, and Kim. In this paper, we build upon this work to prove in much greater generality a topological description of minimizers of the average distance problem conjectured by Buttazzo, Oudet, and Stepanov. In particular, we prove this conjecture in all dimensions in the case originally studied by these authors.
\end{abstract}

\maketitle

\tableofcontents

\section{Introduction}\label{sec:introduction}

In this paper, we will study the minimizers of the \textit{(hard-constraint) average distance problem} introduced by Buttazzo, Oudet, and Stepanov in \cite{Buttazzo02}. Given a compactly supported Borel probability measure $\mu$ on $\mathbb{R}^d$, a cost function $\phi: [0, \infty) \to \mathbb{R}$ with $\phi(0) = 0$, and a Borel $\Sigma \subset \mathbb{R}^d$, define the \textit{average distance functional}
\begin{equation}\label{eq:averagedistancefunctionaldef}
    \mathscr{J}_{\phi}^{\mu}(\Sigma) := \int_{\mathbb{R}^d}\phi(\dist(x, \Sigma)) d\mu(x);
\end{equation}
here $\dist(x, \Sigma) := \inf_{\sigma \in \Sigma} d(x,\sigma)$, and $d(\cdot, \cdot)$ is the Euclidean distance. If we define for $\ell > 0$
\begin{equation}\label{eq:sldefinition}
    \mathcal{S}_{\ell}:= \{\Sigma \subseteq \mathbb{R}^d \ | \ \Sigma \text{ is compact, connected, and }\mathcal{H}^1(\Sigma) \leq \ell \},
\end{equation}
then the average distance problem asks us to study the sets $\Sigma_{\mathrm{opt}} \in \mathcal{S}_{\ell}$ for which
\begin{equation}\label{eq:averagedistanceproblemdefintion}
    \mathscr{J}^{\mu}_{\phi}(\Sigma_{\mathrm{opt}}) = \inf_{\Sigma \in \mathcal{S}_{\ell}}\mathscr{J}^{\mu}_{\phi}(\Sigma).
\end{equation}
Existence of minimizers in the average distance problem is well-known:
\begin{lemma}[Existence of minimizers]
    Assume $\phi$ is lower semicontinuous. Then, there exists $\Sigma_{\mathrm{opt}} \in \mathcal{S}_{\ell}$ such that
    \[
    \mathscr{J}_{\phi}^{\mu}(\Sigma_{\mathrm{opt}}) = \inf_{\Sigma \in \mathcal{S}_{\ell}}\mathscr{J}_{\phi}^{\mu}(\Sigma).
    \]
\end{lemma}
\begin{proof}
    See, for example, \cite{Stepanov06}*{Theorem 4.1}.
\end{proof}

\subsection{Motivation and history}\label{subsec:motivationandhistory}
The original motivation in \cite{Buttazzo02} for studying the average distance problem arose from an optimal transport problem in which the set $\Sigma$ represents a region over which the cost of transporting mass is negligible. The prototypical example is that of a farmer wishing to build an irrigation network in order to water a crop distributed according to $\mu$, where the water source is located within the network. Assuming the cost of transporting water from a point $x \in \mathbb{R}^d$ to a point $y \in \mathbb{R}^d$ is modelled by $\phi(d(x,y))$, while the cost of transport within the network is negligible, an irrigation network of length $\ell$ which results in the lowest total transport cost is modelled by an average distance minimizer $\Sigma$. For a detailed discussion of the relationship between the average distance problem and optimal transport problems, see \cites{Buttazzo02, Buttazzo03} and \cite{Stepanov06}, particularly \cite{Stepanov06}*{Proposition 8.2}.

Over the past two decades, a significant body of work has been amassed studying the properties of minimizers of the average distance problem, a comprehensive account of which can be found in the survey \cite{Lemenant12}. A problem of particular interest is the study of the topological properties of average distance minimizers. In particular, under reasonable conditions, we expect to be able to give the following \textit{complete topological description} of minimizers: minimizers should not contain any loops, should contain only triple junctions, and should have only finitely many endpoints. That these first two properties hold for minimizers was conjectured in 2002 in \cite{Buttazzo02}*{Problem 3.2} and \cite{Buttazzo02}*{Problem 3.3}, respectively. 

This topological description was shown to hold for dimension $d =2$, $\phi(t) = t$, and $\mu << \mathcal{H}^2$ by Buttazzo and Stepanov in 2003 \cite{Buttazzo03}. At the heart of Buttazzo and Stepanov's argument was the \textit{existence of an atom} for average distance minimizers; that is, the existence of a point $\sigma \in \Sigma$ for which 
\begin{equation}\label{eq:atomdefinition}
    \mu(\{x \in \mathbb{R}^d \ | \ \dist(x,\Sigma) = d(x, \sigma) \}) > 0.
\end{equation}
Calling a point $\sigma \in \Sigma$ satisfying \eqref{eq:atomdefinition} an \textit{atom of $\Sigma$} is justified by the fact that if $\pi_{\Sigma}$ is a \textit{closest-point projection} onto the set $\Sigma$ (see \cref{lemma:closestpointprojectionexistence}), then $\sigma$ is an atom of $\Sigma$ precisely when $\sigma$ is an atom of the measure $\nu_{\pi_{\Sigma}} : = (\pi_{\Sigma})_{\#}\mu$ by \cref{cor:uniquenessofclosestpointprojections}. The crucial property afforded by the existence of an atom is the ability to decrease the objective value $\mathscr{J}^{\mu}(\Sigma)$ by $C\epsilon$ when given $\epsilon$ additional $\mathcal{H}^1$-budget, which allows one to perform contradiction arguments by modifying a minimizer with an undesired property in order to gain back $\mathcal{H}^1$-budget, and then using this extra budget to construct a competitor with strictly smaller objective value. 

In 2004, Paolini and Stepanov extended the absence of loops for average distance minimizers to general dimensions in \cite{Stepanov04}*{Theorem 5.6}. Soon after, Stepanov was able to \textit{conditionally} prove the remainder of the  complete topological description in \cite{Stepanov06}*{Theorem 5.5}, under the assumption that minimizers have atoms. In 2013, Lu and Slep\v{c}ev proved the complete topological description with the cost function $\phi(t) = t$ and in dimensions $d \geq 2$ for the weaker \textit{soft-penalty average distance problem} \eqref{eq:softpenaltyadp} in \cite{Slepcev13}*{Lemmas 3.1, 3.2, and 3.4}. However, while the topological characterization for the soft-penalty problem in \cite{Slepcev13} holds for arbitrary compactly supported probability measures $\mu$, a proof of the topological characterization for the hard-constraint problem must use the assumption that $\mu$ is not supported on a $1$-dimensional set in a critical way; see \cref{remark:comparisonofhardconstraintandsoftpenalty}. So, in the following decade, proving existence of an atom in dimension $d > 2$ for the (hard-constraint) average distance problem remained an open problem “of great interest” \cite{Lemenant12}*{Section 8, Problem 3}. 

Concurrently, a significant body of research motivated by statistics and machine learning formed around regularizations of the \textit{principal curves problem} introduced by Hastie and Stuetzle \cite{Hastie89} in 1989. Hastie and Stuetzle aimed to generalize linear principal component analysis to allow for curves instead, and did so by defining \textit{principal curves} to be (smooth) injective curves satisfying the \textit{self-consistency} property: namely, that they are local extrema of the average distance functional with cost function $\phi(t) =t^2$ under continuous perturbations. Difficulties with proving the existence of principal curves even for reasonable densities led to the study of \textit{regularized principal curves}, most notably the \textit{length-constrained principal curves} introduced by K\'{e}gl et. al. \cite{Kegl00} in 2000. The length-constrained principal curves problem is a parameterized version of the average distance problem, and replaces the $\mathcal{H}^1$-measure constraint with a constraint on the arclength of curves:
\begin{equation}\label{eq:length-constrainedprincipalcurves}
    \min_{L(\gamma) \leq \ell}\mathscr{J}^{\mu}_{\phi}(\img \gamma).
\end{equation}
The research on length-constrained principal curves has benefited from this relationship with the average distance problem (see in particular \cites{Delattre17, Lu16}), as has research on other regularized principal curves problems, including \textit{multiple length-constrained principal curves} introduced in \cite{Kirov16}, and \textit{curvature-penalized principal curve/manifold} problems such as \cites{Lu20, Kobayashi24}.

This connection has also informed research on the average distance problem. In the recent paper \cite{Obrie25}, Kobayashi, Kim, and the author proved existence of an atom (and hence the complete topological description) for hard-constraint average distance minimizers in general dimensions, under the assumptions that $\phi(t) = t^p$ for $p =2$ or $p > \frac{1}{2}(3 + \sqrt{5})$ and $\mu$ does not charge 1-dimensional sets. This was made possible by studying the average distance problem using the \textit{barycentre field}, introduced previously by Kobayashi, Hayase, and Kim in \cite{Kobayashi24} to study curvature-penalized principal curves. The barycentre field is essentially the gradient of the average distance functional under continuous perturbations. By adapting Delattre and Fischer's proof of \textit{default of self-consistency} \cite{Delattre17}*{Lemma 3.2} for the length-constrained principal curves problem, Kobayashi, Kim, and the author were able to show that the barycentre field is \textit{nontrivial} under the assumptions outlined above: in particular, this implies that given $\epsilon$ additional $\mathcal{H}^1$-measure budget, one can decrease the objective value by $C \epsilon$. This, in turn, was shown to imply existence of an atom in \cite{Obrie25}*{Theorem 3.2}. 

Due to its close relationship with the existence of an atom \eqref{eq:atomdefinition}, extending the barycentre nontriviality result of \cite{Obrie25} is a problem of great interest. This paper provides a new argument for barycentre nontriviality and existence of an atom which significantly generalizes the results of \cite{Obrie25}. In particular, we prove for the first time in arbitrary dimension the complete topological description of minimizers of the hard-constraint average distance problem in the case $\phi(t) = t$ originally studied by Buttazzo, Oudet, and Stepanov, affirmatively resolving conjectures in \cite{Buttazzo02}. 

\subsection{Main results and outline}

In \cref{sec:thebarycentrefield}, we review the basic theory of the barycentre field as a means of studying the average distance functional, including its fundamental gradient property (\cref{theorem:barycentregradient}). The barycentre field was defined in \cite{Kobayashi24} in the case $\phi(t) = t^p$ for $p \geq 1$; for $\phi \in C^{1}([0, \infty))$, we may define it as follows. Notice that the assumption that $\phi \in C^1([0, \infty))$ implies that $\lim_{t \to 0^+}\phi'(t) < +\infty$: in particular, our results do \textit{not} apply to the cost functions $\phi(t) = t^p$ for $p < 1$. 
 \begin{definition}[Barycentre field]\label{def:introbarycentrefield}
    Let $\Sigma \in \mathcal{S}_{\ell}$, and let $\pi_{\Sigma}$ be a closest-point projection onto $\Sigma$ (see \cref{lemma:closestpointprojectionexistence}). Let $\nu = (\pi_{\Sigma})_{\#}\mu$, and let $(\nu, \{\lambda_{\sigma}\}_{\sigma \in \Sigma})$ be the disintegration of $\mu$ by $\pi_{\Sigma}$. Then, we define the \textit{barycentre field of $\pi_{\Sigma}$} by

    \[
    \mathcal{B}_{\pi_{\Sigma}}(\sigma) := \int_{\pi_{\Sigma}^{-1}\{\sigma\}}\frac{(x - \pi_{\Sigma}(x))}{|x - \pi_{\Sigma}(x)|}\phi'(|x - \pi_{\Sigma}(x)|)d\lambda_{\sigma}(x).
    \]
\end{definition}

In \cref{sec:boundingmassofnoncutpoints}, we prove that the $\nu$-mass of any \textit{noncut point} – that is, a point $\sigma^* \in \Sigma$ such that $\Sigma \setminus \{\sigma^*\}$ is connected – can be bounded from below by the $L^1(\nu)$-norm of the barycentre field (see \cref{theorem:boundingmassofnoncutpoints} for a more quantitative statement):
\begin{theorem}[Bounding the mass of noncut points] \label{theorem:introboundingmassofnoncutpoints}
    Assume $\phi \in C^1([0, \infty))$ is nondecreasing, $\mu(\Sigma) = 0$ for all $\Sigma \subset \mathbb{R}^d$ with $\mathcal{H}^1(\Sigma) < \infty$, and let $\Sigma$ be a solution to the average distance problem. Then, there exists some constant $C > 0$ such that for every noncut point $\sigma^* \in \Sigma$,

    \[
    C \int_{\Sigma}|\mathcal{B}_{\pi_{\Sigma}}(\sigma)|d\nu \leq |\mathcal{B}_{\pi_{\Sigma}}(\sigma^*)|\nu(\sigma^*),
    \]
    where $\nu = (\pi_{\Sigma})_{\#}\mu$.
    In particular, if $\int_{\Sigma}|\mathcal{B}_{\pi_{\Sigma}}(\sigma)|d\nu \ne 0$, then every noncut point of $\Sigma$ is an atom, and $\Sigma$ has only finitely many noncut points. 
\end{theorem}

This generalizes \cite{Obrie25}*{Theorem 3.2}, which proves the same result for $\phi(t) = t^p$, $p \geq 2$. In \cref{sec:softpenatlyatom}, we apply \cref{theorem:introboundingmassofnoncutpoints} to the weaker \textit{soft-penalty average distance problem} defined in \eqref{eq:softpenaltyadp} to prove
\begin{theorem}[Existence of an atom for soft-penalty minimizers]\label{theorem:introexistenceofanatomforsoftpenaltyminimizers}
    Assume $\phi \in C^1([0, \infty))$ is nondecreasing. Suppose that $\mu(\Sigma) = 0$ for all $\Sigma \subseteq \mathbb{R}^d$ with $\mathcal{H}^1(\Sigma) < \infty$. Let $\Sigma$ be a solution to the soft-penalty average distance problem \eqref{eq:softpenaltyadp}. Then, $\Sigma$ has an atom. Assuming additionally that $\phi$ satisfies the condition \eqref{eq:alpha2}, $\Sigma$ satisfies the complete topological description
    \begin{enumerate}
        \item $\Sigma$ contains no loops,
        \item $\Sigma$ has finitely many endpoints, and
        \item $\Sigma$ has finitely many branching points, all of which are triple branchings. 
    \end{enumerate}
    In particular, all of the above holds if we take $\phi(t) = t^p$ for $p \geq 1$. 
\end{theorem}

This generalizes the previous results \cite{Slepcev13}*{Lemmas 3.1-3.4} and \cite{Obrie25}*{Corollary 1.5}, which obtain \cref{theorem:introexistenceofanatomforsoftpenaltyminimizers} under the assumption that $\phi(t) = t^p$ for $p=1$ and $p \geq 2$, respectively. 

In \cref{sec:atomsandbarycentrefield}, we return to our study of the hard-constraint average distance problem \eqref{eq:averagedistanceproblemdefintion}. We begin by proving that existence of an atom implies that the barycentre field has nonzero $L^1(\nu)$-norm, which combined with \cref{theorem:introboundingmassofnoncutpoints} yields (see \cref{cor:nontrivialbarycentrefieldandatoms} for a more robust statement):

\begin{theorem}[Nontrivial barycentre field is equivalent to existence of an atom]\label{cor:intronontrivialbarycentrefieldandatoms}
    Assume $\phi \in C^{1,1}_{\mathrm{loc}}([0, \infty))$ is strictly increasing. Let $\Sigma \in \mathcal{S}_{\ell}$ be optimal, and let $\pi_{\Sigma} \in \Pi_{\Sigma}$. Then, $||\mathcal{B}_{\pi_{\Sigma}}||_{L^1(\nu)} \ne 0$ if and only if there exists some $\sigma \in \Sigma$ such that $\nu(\{\sigma\}) > 0$, where $\nu := (\pi_{\Sigma})_{\#}\mu$.
\end{theorem}
With this characterization of existence of an atom in hand, we proceed in \cref{sec:existenceofatoms} to the main theorem of the paper, the existence of atoms for (hard-constraint) average distance minimizers (see \cref{theorem:existenceofatom}). 

\begin{theorem}[Minimizers have atoms]\label{theorem:introexistenceofatom}
     Assume $\phi \in C^{1,1}_{\mathrm{loc}}([0, \infty))$ is strictly increasing and satisfies \eqref{eq:alpha2}. Assume that $\mu(B_{\epsilon}(x)) = o(\epsilon)$ for each $x \in \mathbb{R}^d$. Then, any minimizer $\Sigma \in \mathcal{S}_{\ell}$ has an atom. 
\end{theorem}

In particular, \cref{theorem:introexistenceofatom} holds for $\phi(t) = t^p$ with $p =1$ or $p \geq 2$. \cref{theorem:introexistenceofatom} generalizes \cite{Obrie25}*{Theorem 3.5}, which proves nontriviality of the barycentre field (and hence existence of an atom) in the case $\phi(t) =t^p$ for $p = 2$ or $p > \frac{1}{2}(3 + \sqrt{5})$. In order to achieve this generalization, we provide a local argument for existence of an atom, in which we compare the mass of two non-atomic noncut points, then use the fact that the barycentre field is trivial to obtain estimates contradicting the minimality of $\Sigma$. This avoids the use of global perturbations as in the proof of \cite{Obrie25}*{Theorem 3.5}, allowing for considerable simplifications. Additionally, \cref{theorem:introexistenceofatom} implies the complete topological description due to Stepanov's conditional result \cite{Stepanov06}*{Theorem 5.5}, see \cref{theorem:stepanovtopologicalcharacterization}.

\begin{theorem}[Complete topological description]\label{theorem:introtopologicalcharacterization}
   Assume $\phi\in C^{1,1}_{\mathrm{loc}}([0, \infty))$ is strictly increasing and satisfies the condition \eqref{eq:alpha2}. Assume that $\mu(B_{\epsilon}(x)) = o(\epsilon)$ for each $x \in \mathbb{R}^d$. Then, any minimizer $\Sigma \in \mathcal{S}_{\ell}$ satisfies the following:
    \begin{enumerate}
        \item $\Sigma$ does not contain any homeomorphic images of $\mathbb{S}^1$, and in particular every noncut point of $\Sigma$ is an endpoint,
        \item the number of endpoints of $\Sigma$ is finite, and
        \item there are finitely many branching points of $\Sigma$, all of which are triple junctions.
    \end{enumerate}
    In particular, the above hold if we take $\phi(t) =t^p$ for $p=1$ or $p \geq 2$. 
\end{theorem}

While (1) was shown to hold previously in \cite{Stepanov04}, \cref{theorem:introtopologicalcharacterization} (2)-(3) generalizes the previous topological characterization given in \cite{Obrie25}. In particular, \cref{theorem:introtopologicalcharacterization} applies to the prototypical case $\phi(t) = t$ of the hard-constraint average distance problem; in this case, \cref{theorem:introtopologicalcharacterization} (3) resolves for the first time the conjecture \cite{Buttazzo02}*{Problem 3.3} in all dimensions $d \geq 2$. Additionally, we hope that our methods may help to understand the topological properties of minimizers of the recently introduced \textit{Wasserstein-}$\mathcal{H}^1$ \textit{problem} \cites{Chambolle23, Machado25}, which has many interesting parallels with the average distance problem.  

Finally, in \cref{sec:branchingrates}, we provide a quantitative version of \cref{theorem:introtopologicalcharacterization}(2) by providing an upper bound on the number of endpoints of an average distance minimizer in terms of the budget $\ell$ and its barycentre field in \cref{prop:initialboundonbranchingrate}.

\subsection{Notation}\label{sec:notation}

Before proceeding, we will record a list of the notation used throughout this paper for ease of reference. For the remainder of the paper, we fix a dimension $d \geq 2$, a compactly supported Borel probability measure $\mu$ on $\mathbb{R}^d$, and a function $\phi: [0, \infty) \to \mathbb{R}$ satisfying $\phi(0) = 0$. Typically, we will assume that $\phi \in C^{1}([0, \infty))$; this means that $\phi'$ is a continuous function on $(0, \infty)$, and that the limit of $\phi'(t)$ as $t \to 0^+$ exists and is a real number. In particular, this assumption excludes the functions $\phi(t) = t^p$ for $p < 1$. We will often assume that $\phi$ satisfies Paolini and Stepanov's condition \eqref{eq:alpha2} from \cite{Stepanov04}:
\begin{quote}
     For every $c > 0$, there is $\lambda = \lambda(c) > 0$ such that 
        \begin{equation}\tag{$\alpha_2$}
            |\phi(s) - \phi(t)| \geq \lambda |s - t|
        \end{equation}
        for any $s, t \in [c, \diam \supp \mu]$,
 \end{quote}
 where $\diam \supp \mu$ is the diameter of the support of $\mu$.
 In particular, \eqref{eq:alpha2} holds for $\phi(t) = t^p$ for all $p \geq 1$.

For a set $\Sigma \subseteq \mathbb{R}^d$, we define the \textit{distance to $\Sigma$} to be
\[
\dist(x, \Sigma) := \inf_{\sigma \in \Sigma}d(x, \sigma),
\]
where $d(\cdot, \cdot)$ is the Euclidean distance. Then, recall that the \textit{average distance functional} \eqref{eq:averagedistancefunctionaldef} is defined for compact sets $\Sigma \subseteq \mathbb{R}^d$ to be 
\[
    \mathscr{J}(\Sigma) = \mathscr{J}_{\phi}^{\mu}(\Sigma) := \int_{\mathbb{R}^d}\phi(\dist(x, \Sigma)) d\mu(x);
\]
reference to $\mu$ and $\phi$ in the notation will typically be suppressed. Given some $s \geq 0$, we denote the \textit{Hausdorff $s$-measure} by $\mathcal{H}^s$. For $\ell > 0$, we define the constraint set 
\[
    \mathcal{S}_{\ell} := \{\Sigma \subseteq \mathbb{R}^d \ | \ \Sigma \text{ is compact, connected, and }\mathcal{H}^1(\Sigma) \leq \ell \},
\]
and we define $\mathcal{S}:= \bigcup_{\ell \geq 0} \mathcal{S}_{\ell}$. Then the \textit{(hard-constraint) average distance problem} \eqref{eq:averagedistanceproblemdefintion} asks us to find the sets $\Sigma_{\mathrm{opt}} \in \mathcal{S}_{\ell}$ such that
\[
    \mathscr{J}^{\mu}_{\phi}(\Sigma_{\mathrm{opt}}) = \inf_{\Sigma \in \mathcal{S}_{\ell}}\mathscr{J}^{\mu}_{\phi}(\Sigma).
\]
Given a solution $\Sigma_{\mathrm{opt}}$ of the average distance problem, we refer to $\mathscr{J}_{\phi}^{\mu}(\Sigma_{\mathrm{opt}})$ as the \textit{objective value of }$\Sigma_{\mathrm{opt}}$, and we will write
\begin{equation}\label{eq:jldefinition}
    j(\ell) := \inf_{\Sigma \in \mathcal{S}_{\ell}}\mathscr{J}(\Sigma)
\end{equation}
for the minimum value in the average distance problem.  
In \cref{sec:softpenatlyatom}, we will also consider the \textit{soft-penalty average distance problem} \eqref{eq:softpenaltyadp}. For $\lambda > 0$, we define the \textit{soft-penalty average distance functional} \eqref{eq:softpenaltyaveragedistancefunctional}
    \[(\mathscr{J}^{\mu}_{\phi})^{\lambda}(\Sigma) := \mathscr{J}^{\mu}_{\phi}(\Sigma) + \lambda \mathcal{H}^1(\Sigma),
\]
for $\Sigma \subseteq \mathbb{R}^d$ compact, then the \textit{soft-penalty average distance problem} asks us to find $\Sigma_{\lambda} \in \mathcal{S}$ satisfying
\eqref{eq:softpenaltyadp}\[    (\mathscr{J}_{\phi}^{\mu})^{\lambda}(\Sigma_{\lambda}) = \inf_{\Sigma \in \mathcal{S}}(\mathscr{J}_{\phi}^{\mu})^{\lambda}(\Sigma).
\]
The soft-penalty average distance problem will only be studied in \cref{sec:softpenatlyatom}; \textit{the average distance problem}, \textit{average distance minimizers}, and \textit{solutions to the average distance problem} will always refer to the hard-constraint average distance problem \eqref{eq:averagedistanceproblemdefintion} and its solutions, unless otherwise specified. 

We denote
\[
\mathscr{F} = \{F : \mathbb{R}^d \to \mathbb{R}^d  \ | \ F\text{ is Borel measureable and } ||F||_{\infty} < \infty\},
\]
where 
\[
||F||_{\infty} = \inf\{a \in \mathbb{R} \ |\ \mu(|F|^{-1}(a, \infty)) = 0 \}.
\]
 For $F \in \mathcal{F}$, denote
\[
\mathcal{I}(F) := \{x \in \mathbf{R}^d \ | \ |F(x) - x| = 0\}. 
\]
Given a compact set $\Sigma \subseteq \mathbb{R}^d$, we define the set of \textit{closest-point projections onto }$\Sigma$ to be
\[
\Pi_{\Sigma} := \{F \in \mathscr{F} \ |\ \img(F) \subset \Sigma \text{ and } \dist(x, \Sigma) = d(x, F(x)) \text{ for all }x \in \mathbb{R}^d \},
\]
then $\Pi_{\Sigma} \ne \emptyset$ by \cref{lemma:closestpointprojectionexistence}. As explained in \cref{remark:extensionofaveragedistancefunctional}, it is useful to extend the average distance functional to take values in $\mathscr{F}$ by defining
\[
    \mathscr{J}(F) = \mathscr{J}_{\phi}^{\mu}(F) := \int_{\mathbb{R}^d}\phi(|x- F(x)|)d\mu(x).
\]
Given $F \in \mathscr{F}$, we will frequently consider the measure 
\[
\nu_F := F_{\#}\mu.
\]
When $F = \pi_{\Sigma}$ and $\pi_{\Sigma}$, $\Sigma$ are clear from context, we will typically suppress reference to $\pi_{\Sigma}$ and write $\nu = \nu_{\pi_{\Sigma}}$. In \cref{sec:boundingmassofnoncutpoints} and \cref{sec:atomsandbarycentrefield}, it will also be helpful to consider the measures \eqref{eq:gammadefinition}
\[
d\gamma_F := \phi'(|x - F(x)|)d\mu
\]
and \eqref{eq:rhodefinition}
\[
\rho_F := F_{\#}(\gamma_F),
\]
as these will frequently show up in bounds. When $F = \pi_{\Sigma}$ and $\pi_{\Sigma}$ is clear from context, we simply write $\gamma = \gamma_{\pi_{\Sigma}}$ and $\rho = \rho_{\pi_{\Sigma}}$. 
For $F \in \mathscr{F}$, we denote \eqref{eq:barycentrekernel}
\[
    \Delta_F(x) := \frac{(x - F(x))}{|x - F(x)|}\phi'(|x - F(x)|).
\]
Letting $(\nu_{F}, \{\lambda_{\sigma}\}_{\sigma \in \mathrm{img}(F)})$ be the disintegration fo $\mu$ by $F$, we define the \textit{barycentre field of }$F$ (\cref{def:barycentrefield})
\begin{align*}
    \mathcal{B}_{F}(\sigma) &:= \int_{F^{-1}\{\sigma\}}\frac{(x - F(x))}{|x - F(x)|}\phi'(|x - F(x)|)d\lambda_{\sigma}(x) \\
    &= \int_{F^{-1}\{\sigma\}} \Delta_F(x)d\lambda_{\sigma}(x).
\end{align*}

Given a connected set $\Sigma \subseteq \mathbb{R}^d$, we say a point $\sigma \in \Sigma$ is a \textit{cut point} if $\Sigma \setminus \{\sigma\}$ is not connected, otherwise $\sigma$ is called a \textit{noncut point}. We refer to a point $\sigma \in \Sigma$ as \textit{an atom of }$\Sigma$ if 
\[
\mu(\{x \in \mathbb{R}^d \ |\ \dist(x, \Sigma) = d(x, \sigma)\}) > 0. 
\]
If $\Sigma \in \mathcal{S}_{\ell}$ is a minimizer of the average distance problem and $\phi \in C^1([0, \infty))$ is strictly increasing, then the \textit{ambiguous locus} (see \cref{def:ambiguouslocus}) of $\Sigma$ is $\mu$-negligible by \cref{prop:negligibilityofambiguouslocus}, and thus all of the closest-point projections onto $\Sigma$ are equal $\mu$-a.e. (see \cref{cor:uniquenessofclosestpointprojections}). In particular, $\sigma \in \Sigma$ is an atom of $\Sigma$ if and only if $\nu_{\pi_{\Sigma}}\{\sigma\} > 0$, so we will use the terms “atom of $\Sigma$” and “atom of $\nu$” for an average distance minimizer interchangeably. 

Given a Borel $A \subseteq \mathbb{R}^d$, we write $\chi_{A}$ for the indicator function of $A$, that is
\[
\chi_{A}(x) = \bigg \{ \begin{matrix}
    1, & x \in A, \\
    0, & x \notin A.
\end{matrix}
\]
Often, we will consider the integral of the function $\chi_{A}\phi(\dist(x, \Sigma))$ for some Borel $\Sigma \subseteq \mathbb{R}^d$; we denote this quantity by 
\begin{equation}\label{eq:restrictedadfnotation}
    \mathscr{J}|_{A}(\Sigma) =\mathscr{J}^{\mu}_{\phi}\mid_{A}(\Sigma) := \int_{A}\phi(\dist(x, \Sigma))d\mu(x),
\end{equation}
and define $\mathscr{J}|_A(F)$ for $F \in \mathscr{F}$ similarly.

\section{The barycentre field}\label{sec:thebarycentrefield}

In this section, we give a generalized definition of the \textit{barycentre field} originally introduced by Kobayashi, Hayase, and Kim in \cite{Kobayashi24}. We then establish the basic theory of the barycentre field for the average distance functional, generalizing the results of \cite{Obrie25}*{Section 2}, as well as providing some general approximation results for the average distance functional. 

\subsection{Preliminary results}\label{sec:preliminaryresults}

A crucial object which enables our study of the barycentre field are the \textit{closest-point projections} onto a compact set $\Sigma \subseteq \mathbb{R}^d$. Since the map
\[
x \mapsto \{\sigma \in \Sigma \ | \ d(x, \sigma) = \dist(x, \Sigma)\}
\]
is in general multivalued, we recall the foundational result \cite{Bertsekas78}*{Proposition~7.33} which allows us to take \textit{measurable selections} of this many-valued function.

\begin{lemma}[Measurable selection, \cite{Bertsekas78}*{Proposition~7.33}]\label{lemma:measurableselection} Let $X$ be a metrizable space, $Y$ a compact metrizable space, $D$ a closed subset of $X \times Y$, and let $f : D \to \mathbb{R} \cup \{- \infty, \infty\}$ be lower semicontinuous. Let $f^*: \pi_X(D) \to \mathbb{R}\cup \{- \infty, \infty\}$ be given by
  \[
    f^*(x) = \min_{(x, y) \in \pi_X^{-1}\{x\} \cap D}f(x, y),
  \]
  where $\pi_X: X \times Y \to X, (x, y) \mapsto x$ is the projection. Then, $\pi_X(D)$ is closed in $X$, $f^*$ is lower semicontinuous, and there exists a Borel-measurable function $\varphi: \pi_X(D) \to Y$ such that $\{(x, \varphi(x)) \ | \ x \in \pi_X(D)\} \subseteq D$ and $f(x, \varphi(x)) = f^*(x)$ for all $x \in \pi_X(D)$.
\end{lemma}

Using measurable selection, we may establish the existence of closest-point projections onto a compact set $\Sigma$.

\begin{lemma}[Existence of closest-point projection]\label{lemma:closestpointprojectionexistence} For any compact $\Sigma \subseteq \mathbb{R}^d$ with $\Sigma \ne \emptyset$, there exists a Borel measurable map $\pi_{\Sigma}: \mathbb{R}^d \to \Sigma$ such that
  \begin{align*}
    \hbox{$\dist(x, \Sigma) = | x - \pi_{\Sigma}(x)|$ for all $x \in \mathbb{R}^d$.}
  \end{align*}
\end{lemma}

\begin{proof}
  Consider the closed set 
  \[D = \{(x, \sigma) \in \mathbb{R}^d \times \Sigma \ | \ \dist(x, \Sigma) = |x - \sigma|\},\]
  and define $f: D \to \mathbb{R}$ by $f(x, \sigma) = |x - \sigma|$. Then, $f$ is continuous, and in particular lower semicontinuous. So, by measurable selection, there exists a Borel-measurable function $\pi_{\Sigma}: \mathbb{R}^d \to \Sigma$ such that $\dist(x, \Sigma) = |x - \pi_{\Sigma}(x)|$ for all $x \in \mathbb{R}^d$.
\end{proof}

\begin{definition}\label{def:closest-proj}
  We will refer a map satisfying the conclusion of \cref{lemma:closestpointprojectionexistence} as a \textit{closest-point projection} onto $\Sigma$. Given some compact and nonempty $\Sigma \subseteq \mathbb{R}^d$, we define the set of closest-point projections onto $\Sigma$ by
  \[
    \Pi_{\Sigma} := \{\pi : \mathbb{R}^d \to \Sigma \ | \ \pi \text{
      is measurable and }\dist(x, \Sigma) = |x - \pi(x)| \text{ for
      all }x \in \mathbb{R}^d \}.
  \]
  Then, by \cref{lemma:closestpointprojectionexistence}, we know that $\Pi_{\Sigma}$ is nonempty.
\end{definition}

Closest-point projections onto a compact set $\Sigma$ are in general \textit{not} unique $\mu$-a.e.; consider the case when $\Sigma = S^{d-1}$, and $\mu = \delta_0$ is a Dirac mass at the origin. This is not the case for average distance minimizers: in \cref{cor:uniquenessofclosestpointprojections}, we will use the barycentre field to prove that any two closest-point projections onto a minimizer of the average distance problem are equal $\mu$-a.e.

\subsection{A general first-order approximation}\label{sec:twogeneralapproximationtheorems}

A main theme throughout this paper is the problem of finding bounds on the difference 
\[
\mathscr{J}(\Sigma) - \mathscr{J}(\Sigma'),
\]
where $\Sigma, \Sigma' \in \mathcal{S}$. For the cases in which we are interested, $\Sigma'$ will typically be obtained from $\Sigma$ through a modification which changes the $\mathcal{H}^1$-measure by some small amount $\epsilon$. The examples to keep in mind are that of a \textit{continuous global modification}, in which $\Sigma' = (\mathrm{id} + \epsilon \xi)(\Sigma)$ for some continuous function $\xi$, and that of a \textit{local modification}, in which $\Sigma' = \Sigma \cup K_{\epsilon}$ or $\Sigma' = \Sigma \setminus K_{\epsilon}$ for some set $K_{\epsilon}$ of length $\mathcal{H}^1(K_{\epsilon}) = \epsilon$.

When trying to bound this difference, it is often easier to work with a well-chosen Borel function $F: \mathbb{R}^d \to \mathbb{R}^d$ such that $\mathrm{img}(F) \subset \Sigma'$ instead of a closest-point projection onto $\Sigma'$. For this reason, we will extend the average distance function to take values in the set
\[
    \mathscr{F} := \{F : \mathbb{R}^d \to \mathbb{R}^d \ | \ ||F||_{\infty} < \infty\},
\]
by defining
\[
    \mathscr{J}(F) = \mathscr{J}_{\phi}^{\mu}(F) := \int_{\mathbb{R}^d}\phi(|x- F(x)|)d\mu(x)
    \]
for $F \in \mathscr{F}$. 

\begin{remark}\label{remark:extensionofaveragedistancefunctional}
    Let $\Sigma \subset \mathbb{R}^d$ be compact, and suppose $F \in \mathscr{F}$ is such that $\img F \subset \Sigma$. Then, for any $\pi_{\Sigma} \in \Pi_{\Sigma}$, we have that 
    \[
    |x - F(x)| \geq \dist(x, \Sigma) = |x - \pi_{\Sigma}(x)|
    \]
    for all $x \in \mathbb{R}^d$. In particular, if $\phi$ is nondecreasing, then we have that 
    \begin{equation}\label{eq:projectionapproximationinequality}
    \mathscr{J}_{\phi}^{\mu}(F) \geq \mathscr{J}_{\phi}^{\mu}(\pi_{\Sigma}) = \mathscr{J}_{\phi}^{\mu}(\Sigma).
    \end{equation}
    Thus, in order to bound $\mathscr{J}(\Sigma) - \mathscr{J}(\Sigma')$ from below, it suffices to bound 
    \[
    \mathscr{J}(\pi_{\Sigma}) - \mathscr{J}(F)
    \]
    from below for some $F \in \mathscr{F}$ satisfying $\img F \subset \Sigma'$. This is why we are interested in studying the extension of the average distance functional to $\mathscr{F}$: in particular, our general first-order approximation result \cref{prop:1generalfirstorderapproximationtheorem1} will be stated in terms of elements of $\mathscr{F}$ instead of in terms of elements of $\mathcal{S}$. 
\end{remark}

Before proceeding to \cref{prop:1generalfirstorderapproximationtheorem1}, we will prove a preliminary lemma. Recall that when $\phi \in C^1([0, \infty))$ and $a \in (0, \infty)$, the mean value theorem gives the first-order approximation
\[
\frac{\phi(a + \epsilon) - \phi(a)}{\epsilon} = \phi'(a) + O(\epsilon).
\]
We wish to extend this approximation to the average distance functional; to do so we require a uniform version of this approximation. Given $F \in \mathscr{F}$, we denote
    \[
    \mathcal{I}(F) := \{x \in \mathbb{R}^d \ | \ F(x) = x\}.
    \]

  \begin{lemma}[Uniform mean value approximation]\label{lemma:phitechnicalapproximation}
      Assume 
      $\phi \in C^1([0, \infty))$. Let $\{F_n\}_{n \in \mathbb{N}} \subseteq \mathscr{F}$, and $F \in \mathscr{F}$ such that $||F_n - F||_{\infty} \to 0$ as $n \to \infty$. Then, for all $x \in \mathbb{R}^d \setminus (\mathcal{I}(F) \cup \bigcup_{n \in \mathbb{N}}\mathcal{I}(F_n))$, there exists a sequence $\{c_n^x\}_{n \in \mathbb{N}} \subseteq (0, \infty)$ such that $c_n^x \to |x - F(x)|$ as $n \to \infty$, and \begin{equation}\label{eq:phitechnicalapproximation}
          \begin{split}
              &\phi(|x - F(x)|) - \phi(|x - F_n(x)|) \\
              & \qquad = (F_n(x) - F(x))\cdot \frac{(x - F(x))}{|x - F(x)|}\phi'(c_n^x)  + O(||F_n - F||_{\infty}^2).
          \end{split}
      \end{equation}
      Moreover, there exists some uniform choice of $C > 0$ and $N \in \mathbb{N}$ such that for $\mu$-a.e. $x \in \mathbb{R}^d \setminus (\mathcal{I}(F) \cup \bigcup_{n \in \mathbb{N}}\mathcal{I}(F_n))$, $c_n^x \leq C$ for all $n \geq N$.
  \end{lemma}
  
  \begin{proof}
      Fix $x \in \mathbb{R}^d \setminus (\mathcal{I}(F) \cup \bigcup_{n \in \mathbb{N}}\mathcal{I}(F_n))$, and fix $n \in \mathbb{N}$. Since $x \notin \mathcal{I}(F) \cup \bigcup_{n \in \mathbb{N}}\mathcal{I}(F_n)$, we know that $|x - F(x)| \ne 0$ and $|x - F_n(x)| \ne 0$. Therefore, since $\phi\in C^1([0, \infty))$, we know by the mean value theorem that there exists $c_n^x \in [\min\{|x - F(x)|, |x - F_n(x)|\}, \max\{|x- F(x)|, |x - F_n(x)|\}]$ such that
      \[
      \phi(|x - F(x)|) - \phi(|x - F_n(x)|) = (|x - F(x)| - |x - F_n(x)|)\phi'(c_n^x).
      \]
       Since 
        \[
        \min\{|x - F(x)|, |x- F_n(x)|\} \leq c_n^x \leq \max\{|x - F(x)| , |x - F_n(x)|\}
        \]
        and $|x - F_n(x)| \to |x - F(x)|$ as $n \to \infty$, we see that $c_n^x \to |x - F(x)|$ as $n \to \infty$. 
        
        Now, notice that 
        \begin{align*}
            &|x - F_n(x)| \\
            & \quad = |x- F(x)| \left(1 - 2(F_n(x)- F(x))\cdot \frac{(x - F(x))}{|x - F(x)|^2} + \frac{|F(x) - F_n(x)|^2}{|x - F(x)|^2}\right)^{\frac{1}{2}} \\
            & \quad = |x - F(x)| - (F_n(x) - F(x))\cdot \frac{(x - F(x))}{|x - F(x)|} + O(||F_n - F||_{\infty}^2),
        \end{align*}
        and thus
        \begin{align*}
            \phi(|x - F(x)|) - \phi(|x - F_n(x)|) &= (F_n(x) - F(x))\cdot \frac{(x - F(x))}{|x - F(x)|}\phi'(c_n^x) \\
            & \qquad + \phi'(c_n^x)O(||F_n - F||_{\infty}^2).
        \end{align*}
        Since $\phi'$ is continuous on $[0, \infty)$, we have that 
        \[\lim_{n \to \infty}\phi'(c_n^x) = \phi'(|x - F(x)|),\]
        so we may absorb $\phi'(c_n^x)$ into the $O(||F_n - F||_{\infty}^2)$ term. This proves \eqref{eq:phitechnicalapproximation} holds. 

        Finally, notice that for each $x \in \supp \mu,$ we have
        \[
        |x - F_n(x)| \leq \diam \supp \mu + ||F_n||_{\infty}.
        \]
        Since $||F - F_n||_{\infty} \to 0$ as $n \to \infty$, we may take $N \in \mathbb{N}$ such that $||F_n||_{\infty} \leq ||F||_{\infty} + 1$ for all $n \geq N$. So, taking $C = \diam \supp \mu + ||F||_{\infty} + 1$, we see that $C$ and $N$ are independent of choice of $x$, and for all $x \in \supp \mu$, for all $n \geq N$, 
        \[
        c_n^x \leq \max\{|x- F(x)|, |x- F_n(x)|\}\leq C,
        \]
        as claimed.
  \end{proof}

For the sake of concision, given $F \in \mathscr{F}$, denote
\begin{equation}\label{eq:barycentrekernel}
    \Delta_F(x) := \frac{(x - F(x))}{|x - F(x)|}\phi'(|x - F(x)|).
\end{equation}
Now, we prove our general first-order approximation.
  
\begin{proposition}[General first-order approximation]\label{prop:1generalfirstorderapproximationtheorem1}
       Assume $\phi \in C^1([0, \infty))$. Let $\{F_n\}_{n \in \mathbb{N}} \subseteq \mathscr{F}$ and $F \in \mathscr{F}$. Suppose that $||F - F_n||_{\infty} \to 0$ as $n \to \infty$, and denote $\epsilon_n = ||F - F_n||_{\infty}$. Suppose that $\mu(\mathcal{I}(F)) = \mu(\mathcal{I}(F_n)) = 0$ and $\epsilon_n > 0$ for each $n \in \mathbb{N}$. Then, \begin{equation}\label{liminffirstorderapproximation}
            \liminf_{n \to \infty}\frac{\mathscr{J}(F) - \mathscr{J}(F_n)}{\epsilon_n} \geq \int_{\mathbb{R}^d} \liminf_{n \to \infty} \frac{(F_n(x) - F(x))}{\epsilon_n} \cdot \Delta_{F}(x) d\mu,
        \end{equation}
        and \begin{equation}\label{limsupfirstorderapproximation}
            \limsup_{n \to \infty}\frac{\mathscr{J}(F) - \mathscr{J}(F_n)}{\epsilon_n} \leq \int_{\mathbb{R}^d} \limsup_{n \to \infty} \frac{(F_n(x) - F(x))}{\epsilon_n} \cdot \Delta_{F}(x) d\mu.
        \end{equation}
    \end{proposition}

    \begin{proof}
        We will only verify equation (\ref{liminffirstorderapproximation}), as (\ref{limsupfirstorderapproximation}) follows by an analogous argument. By \cref{lemma:phitechnicalapproximation}, we have for each $x \in \mathbb{R}^d \setminus (\mathcal{I}(F) \cup \bigcup_{n \in \mathbb{N}}\mathcal{I}(F_n))$ that
        \[
        \liminf_{n \to \infty} \frac{\phi(|x - F(x)|) - \phi(|x - F_n(x)|)}{\epsilon_n} = \liminf_{n \to \infty} \frac{(F_n(x) - F(x))}{\epsilon_n}\cdot \Delta_{F}(x).
        \]
        Since $\mu(\mathcal{I}(F) \cup \bigcup_{n \in \mathbb{N}}\mathcal{I}(F_n)) = 0$ by assumption, we conclude that this equality holds $\mu$-a.e. Moreover, taking $C$ and $N$ as in \cref{lemma:phitechnicalapproximation}, we have that for $\mu$-a.e. $x \in \mathbb{R}^d$ and all $n \geq N$ that $c_n^x \leq C$. 
        Since $\phi'$ is continuous on $[0, \infty)$, we see that it attains a maximum $M$ on the interval $[0,C]$, so by Lemma \ref{lemma:phitechnicalapproximation} we have for all sufficiently large $n$ that 
        \[
        |\inf_{m \geq n}(\frac{\phi(|x - F(x)|) - \phi(|x - F_n(x)|)}{\epsilon_n})| \leq 2M
        \]
        for $\mu$-a.e. $x \in \mathbb{R}^d$. Since $\mu$ is finite, the constant function $x \mapsto 2M$ has bounded integral, so by the dominated convergence theorem we see that
        \begin{align*}   &\int_{\mathbb{R}^d}\liminf_{n \to \infty} \frac{(F_n(x) - F(x))}{\epsilon_n} \cdot \Delta_F(x) d\mu \\
        &\qquad = \lim_{n \to \infty} \int_{\mathbb{R}^d} \inf_{m \geq n}(\frac{\phi(|x - F(x)|) - \phi(|x - F_n(x)|)}{\epsilon_n})d\mu \\
        & \qquad \leq \liminf_{n \to \infty}\frac{\mathscr{J}(F) - \mathscr{J}(F_n)}{\epsilon_n},
        \end{align*}
        proving \eqref{liminffirstorderapproximation}.
    \end{proof}

\subsection{The barycentre field}\label{subsec:thebarycentrefield}

    We are now ready to introduce the barycentre field, and its fundamental property as the “negative gradient” of $\mathscr{J}$ under continuous perturbations.

    \begin{definition}[Barycentre field]\label{def:barycentrefield}
    Let $F \in \mathscr{F}$, and let $\nu_F = F_{\#}\mu$. Let $(\nu_F, \{\lambda_{\sigma}\}_{\sigma \in \img(F)})$ be the disintegration of $\mu$ by $F$. Then, we define the \textit{barycentre field of $F$} by

    \begin{align*}
        \mathcal{B}_{F}(\sigma) &:= \int_{F^{-1}\{\sigma\}}\frac{(x - F(x))}{|x - F(x)|}\phi'(|x - F(x)|)d\lambda_{\sigma}(x) \\ &= \int_{F^{-1}\{\sigma\}} \Delta_F(x)d\lambda_{\sigma}(x).
    \end{align*}
\end{definition}

Now, we generalize the relevant inequality in \cite{Kobayashi24}*{Theorem 4.8} and \cite{Obrie25}*{Proposition 2.11} to the case of nondecreasing $\phi \in C^{1}([0, \infty))$. A version of this result is also given in the case $\phi(t) = t$ for the soft-penalty average distance functional by Buttazzo, Manini, and Stepanov \cite{Buttazzo09}*{Theorem 2.1}. 
\begin{proposition}[Gradient interpretation of the barycentre field]\label{theorem:barycentregradient}
    Assume $\phi \in C^1([0, \infty))$ is nondecreasing.
    Let $\Sigma \subseteq \mathbb{R}^d$, and let $\xi: \Sigma \to \mathbb{R}^d$ be continuous with $||\xi||_{\infty} \leq 1$. For $\epsilon > 0$, define $\Sigma_{\epsilon, \xi} = \{\sigma + \epsilon \xi (\sigma) \ | \ \sigma \in \Sigma\}$. Suppose that $\mu(\Sigma) = 0$. Then, for every $\pi_{\Sigma} \in \Pi_{\Sigma}$,
\begin{equation}\label{eq:barycentregradient}
    \lim_{\epsilon \to 0^+} \frac{\mathscr{J}(\Sigma) - \mathscr{J}(\Sigma_{\epsilon, \xi})}{\epsilon} \geq   \int_{\Sigma} \xi(\sigma) \cdot \mathcal{B}_{\pi_{\Sigma}}(\sigma)d\nu_{\pi_{\Sigma}}(\sigma).
\end{equation}
\end{proposition}

\begin{proof}
    Let $\pi_{\Sigma} \in \Pi_{\Sigma}$ be arbitrary. Let $\{\epsilon_n\}_{n \in \mathbb{N}}\subseteq (0, \infty)$ with $\epsilon_n \to 0$ as $n \to \infty$. Take $F_n = \pi_{\Sigma} + \epsilon_n \xi\circ \pi_{\Sigma}$, then $\img F_n \subseteq \Sigma_{\epsilon_n, \xi}$ and so $\mathscr{J}(\Sigma_{\epsilon, \xi}) \leq \mathscr{J}(F_n)$ since $\phi$ is nondecreasing. Notice that $||\pi_{\Sigma} - F_n||_{\infty} = \epsilon_n ||\xi||_{\infty} \leq \epsilon_n$. So, we may apply \cref{prop:1generalfirstorderapproximationtheorem1} to see that

    \begin{align*}
        \liminf_{n \to \infty}\frac{\mathscr{J}(\Sigma)- \mathscr{J}(\Sigma_{\epsilon, \xi})}{\epsilon_n} &\geq \liminf_{n \to \infty} \frac{\mathscr{J}(\pi_{\Sigma}) - \mathscr{J}(F_n)}{\epsilon_n} \\
        &\geq \int_{\mathbb{R}^d}\xi(\pi(x))\cdot \Delta_F(x)d\mu \\
        &= \int_{\Sigma}\xi(\sigma)\cdot \mathcal{B}_{\pi_{\Sigma}}(\sigma)d\nu_{\pi_{\Sigma}}(\sigma).
    \end{align*}
    Thus, since $\{\epsilon_n\}_{n \in \mathbb{N}}$ was arbitrary, we conclude that
    \[
    \liminf_{\epsilon \to 0^+}\frac{\mathscr{J}(\Sigma) - \mathscr{J}(\Sigma_{\epsilon, \xi})}{\epsilon}\geq \max_{\pi_{\Sigma} \in \Pi_{\Sigma}}\int_{\Sigma}\xi(\sigma)\cdot \mathcal{B}_{\pi_{\Sigma}}d\nu_{\pi_{\Sigma}}(\sigma).
    \]
\end{proof}

\begin{remark}
    In fact, we expect that the bound \eqref{eq:barycentregradient} is sharp: \cite{Kobayashi24}*{Theorem 4.8} shows under the same assumptions as \cref{theorem:barycentregradient} that for $\phi(t) = t^p$, $p \geq 1$,
    \[
    \lim_{\epsilon \to 0^+}\frac{\mathscr{J}(\Sigma) - \mathscr{J}(\Sigma_{\epsilon, \xi})}{\epsilon} = \max_{\pi_{\Sigma} \in \Pi_{\Sigma}}\int_{\Sigma}\xi(\sigma)\cdot \mathcal{B}_{\pi_{\Sigma}}(\sigma)d\nu_{\pi_{\Sigma}}(\sigma).
    \]
    This result should likely extend to the case of $\phi \in C^1([0, \infty))$ nondecreasing, but as the inequality \eqref{eq:barycentregradient} is sufficient for the purposes of this paper, we will not discuss this equality further. 
\end{remark}

We now continue our analogy between the barycentre field and the negative gradient of a function by interpreting another key property of gradients in terms of the barycentre field. Should the negative gradient of a function be nonzero, it will point in the direction of greatest decrease, and so by perturbing by $\epsilon$ in the direction of the negative gradient, we can decrease the value of the function by order $\epsilon$. We wish to prove a similar property for the barycentre field. First, we will define what it means for $\mathscr{J}$ to have “nonzero gradient” at $F \in \mathscr{F}$.
\begin{definition}\label{def:nontrivialbarycentrefield}
    We say that $F$ has \textit{trivial barycentre field} if 
    \[
    \nu_{F}\{\sigma \in \img(F) \ |\ \mathcal{B}_{F}(\sigma) \ne 0\} = 0. 
    \]
    Otherwise, we say $F$ has \textit{nontrivial barycentre field}. Similarly, we say a compact set $\Sigma \subseteq \mathbb{R}^d$ has \textit{trivial barycentre field} or satisfies the \textit{(generalized) self-consistency property} if $\pi_{\Sigma}$ has trivial barycentre field for all $\pi_{\Sigma} \in \Pi_{\Sigma}$. Otherwise, we say $\Sigma$ has \textit{nontrivial barycentre field}.
\end{definition}

\begin{remark}\label{remark:generalizedselfconsistencyproperty}
    Notice that a curve $\gamma: [0,1] \to \mathbb{R}^d$ satisfies the self-consistency property used by Hastie and Steutzle to define principal curves \cite{Hastie89} if and only if $\img(\gamma)$ satisfies the generalized self-consistency property defined in \cref{def:nontrivialbarycentrefield}, in the case $\phi(t) = t^2$. Indeed, $\pi \in \Pi_{\img(\gamma)}$ has trivial barycentre field for $\phi(t) = t^2$ if and only if for $\nu_{\pi}$-a.e. $\gamma(s) \in \img(\gamma)$,
    \[
    \int_{\pi^{-1}(\gamma(s))}(x - \gamma(s))d\lambda_{\gamma(s)}(x) = \int_{\pi^{-1}(\gamma(s))} x d\lambda_{\gamma(s)}(x) - \gamma(s) = 0,
    \]
    which is precisely the self-consistency property of Hastie and Steutzle \cite{Hastie89}*{Definition 1}. The barycentre field thus gives a natural way to generalize the self-consistency property, justifying our terminology in \cref{def:nontrivialbarycentrefield}.
\end{remark}

In order to use \cref{theorem:barycentregradient} to perturb in the “direction” of the barycentre field, we need to find a continuous perturbation $\xi$ for which the right hand side of \eqref{eq:barycentregradient} is positive. In general, we should not expect the barycentre field to be continuous, as can be seen in \cite{Kobayashi24}*{Counterexample 4.7}; this means that it is necessary to approximate the barycentre field with a continuous perturbation. In fact, we will approximate $\mathcal{B}_{\pi_{\Sigma}}$ by a Lipschitz perturbation: this allows us to control the change in $\mathcal{H}^1$-budget as we perturb $\Sigma$.

\begin{proposition}[Approximation of  $\mathcal{B}_{\pi_{\Sigma}}$]\label{prop:lipschitzapproximation}
  Assume $\phi \in C^1([0, \infty))$ is nondecreasing.
  Suppose that $\pi_{\Sigma} \in \Pi_\Sigma$ has nontrivial barycentre
  field. Then, there exists a Lipschitz map $\xi : \mathbb{R}^d \to
  \mathbb{R}^d$ such that
  \[\int_{\Sigma} \xi(\sigma) \cdot
    \mathcal{B}_{\pi_{\Sigma}}(\sigma) d \nu_{\pi_{\Sigma}}(\sigma) >
    \frac{1}{2}\int_{\Sigma}|\mathcal{B}_{\pi_{\Sigma}}(\sigma) |^2 d
    \nu_{\pi_{\Sigma}}(\sigma) > 0.\]
\end{proposition}
\begin{proof}
    Let 
    \[
    M = \max_{t \in [0, \diam(\supp \mu)]}\phi'(t),
    \]
    then since $\phi \in C^1([0, \infty))$, $M < \infty$. Notice that $|\mathcal{B}_{\pi_{\Sigma}}| \leq M$, and thus $\mathcal{B}_{\pi_{\Sigma}} \in L^2(\Sigma, \nu)$. But Lipschitz functions are dense in $L^2(\Sigma, \nu)$ by \cite{Obrie25}*{Lemma 2.14}, so approximating $\mathcal{B}_{\pi_{\Sigma}}$ by Lipschitz functions yields the desired result. 
\end{proof}

In particular,
\begin{corollary}[Right derivative bound on $j$]\label{cor:rightderivativeboundonj}
    Assume $\phi \in C^1([0, \infty))$ is nondecreasing. Let $\ell > 0$, and suppose there exists a solution $\Sigma \in \mathcal{S}_{\ell}$ to the average distance problem with nontrivial barycentre field. Then, there exists some $C$ such that
    \[
    \liminf_{\epsilon \to 0^+}\frac{j(\ell) - j(\ell + \epsilon)}{\epsilon} \geq C > 0.
    \]
\end{corollary}
\begin{proof}
    Take $\pi_{\Sigma} \in \Pi_{\Sigma}$ with nontrivial barycentre field, and take $\xi$ as in \cref{prop:lipschitzapproximation}. Let $L > 0$ be a Lipschitz constant for $\xi$, and consider $\zeta = \frac{1}{\ell L}\xi$. Then, $\Sigma_{\epsilon, \zeta} \in S_{\ell+\epsilon}$, and by \cref{theorem:barycentregradient}, 
    \begin{align*}
        \liminf_{\epsilon \to 0^+}\frac{j(\ell) - j(\ell + \epsilon)}{\epsilon} &\geq \lim_{\epsilon \to 0^+}\frac{\mathscr{J}(\Sigma) - \mathscr{J}(\Sigma_{\epsilon, \zeta})}{\epsilon} \\&\geq \int_{\Sigma}\zeta(\sigma)\cdot \mathcal{B}_{\pi_{\Sigma}}(\sigma)d\nu_{\pi_{\Sigma}}(\sigma)\\  &\geq \frac{1}{2L\ell}\int_{\Sigma}|\mathcal{B}_{\pi_{\Sigma}}(\sigma)|^2 d\nu_{\pi_{\Sigma}}(\sigma) > 0.
    \end{align*}
\end{proof}

\subsection{Negligibility of the ambiguous locus}\label{sec:negligibilityofambiguouslocus}

Now, we will give our first example of the usefulness of the barycentre field by proving that the \textit{ambiguous locus} (the set of points which do not unique closest point in $\Sigma$) is $\mu$-null for sets $\Sigma \subseteq \mathbb{R}^d$ which are local minima of the average distance functional under translation. The results of this section generalize those in \cite{Obrie25}*{Section 2.4}, which in turn generalizes the result of Delattre and Fischer \cite{Delattre17}*{Proposition 3.1}. 

\begin{definition}[Ambiguous locus]\label{def:ambiguouslocus}
  Let $\Sigma \subseteq \mathbb{R}^d$ be compact and nonempty, and for each $x \in \mathbb{R}^d$ consider the set $$\mathcal{P}_{\Sigma}(x) = \{\sigma \in \Sigma \ | \ d(x, \sigma) = \dist(x, \Sigma)\}.$$ We define the \textit{ambiguous locus} of $\Sigma$ to be
  \[
    \mathcal{A}_{\Sigma} = \{x \in \mathbb{R}^d \ | \ \#\mathcal{P}_{\Sigma}(x) > 1\}.
  \]
\end{definition}

First, we show that if the \textit{net barycentre field} is nonzero, then we can always decrease the objective value by translating in the direction of the net barycentre field. 
In particular, if $\Sigma$ is a minimum of the average distance functional over the set of its translates by vectors in a neighbourhood of $0$ in $\mathbb{R}^d$, then the net barycentre field of $\pi_{\Sigma}$ equals $0$ for any $\pi_{\Sigma} \in \Pi_{\Sigma}$. 

\begin{lemma}[Net barycentre field of a minimizer is zero]\label{lemma:netbarycentrefieldiszero}
    Assume $\phi \in C^1([0, \infty))$ is nondecreasing. Let $\Sigma \subseteq \mathbb{R}^d$ be measurable and nonempty, and let $\pi_{\Sigma} \in \Pi_{\Sigma}$. Let
    \begin{equation}\label{eq:netbarycentrefielddef}
        \mathcal{B}_{\pi_{\Sigma}}^{\mathrm{net}} := \int_{\Sigma}\mathcal{B}_{\pi_{\Sigma}}(\sigma)d\nu_{\pi_{\Sigma}}(\sigma) = \int_{\mathbb{R}^d}\frac{(x - \pi_{\Sigma}(x))}{|x - \pi_{\Sigma}(x)|}\phi'(|x - \pi_{\Sigma}(x)|)d\mu
    \end{equation}
    be the \textit{net barycentre field} of $\pi_{\Sigma}$, and define $\Sigma_{\epsilon} : = \Sigma + \epsilon \mathcal{B}_{\pi_{\Sigma}}^{\mathrm{net}}$. Then, 
    \[
    \lim_{\epsilon \to 0^+}\frac{\mathscr{J}(\Sigma) - \mathscr{J}(\Sigma_{\epsilon})}{\epsilon} \geq |\mathcal{B}_{\pi_{\Sigma}}^{\mathrm{net}}|^2.
    \] 
    In particular, if $0$ is a local minimum of the map $\mathbb{R}^d \to \mathbb{R}, a \mapsto \mathscr{J}(\Sigma + a)$, then $\mathcal{B}_{\pi_{\Sigma}}^{\mathrm{net}} = 0$.
\end{lemma}
\begin{proof}
    Applying \cref{theorem:barycentregradient} with $\xi = \mathcal{B}_{\pi_{\Sigma}}^{\mathrm{net}}$, we get
    \[
    \lim_{\epsilon \to 0^+}\frac{\mathscr{J}(\Sigma) - \mathscr{J}(\Sigma_{\epsilon})}{\epsilon} \geq \int_{\Sigma} \mathcal{B}_{\pi_{\Sigma}}^{\mathrm{net}}\cdot \mathcal{B}_{\pi_{\Sigma}}(\sigma)d\nu(\sigma) = |\mathcal{B}_{\pi_{\Sigma}}^{\mathrm{net}}|^2.
    \]
\end{proof}

Using \cref{lemma:netbarycentrefieldiszero} to play the role of \cite{Delattre17}*{Remark 2}, we can generalize Delattre and Fischer's proof of \cite{Delattre17}*{Proposition 3.1}.

\begin{proposition}[Negligibility of the ambiguous locus]\label{prop:negligibilityofambiguouslocus}
  Assume $\phi \in C^1([0, \infty))$ is strictly increasing. Suppose that $\mathscr{J}(\Sigma + a) \geq \mathscr{J}(\Sigma)$ for all $a$ in some open neighbourhood of $0$. Then, $\mu(\mathcal{A}_{\Sigma}) = 0$.
\end{proposition}
\begin{proof}
  Suppose for the sake of contradiction that $\mu(\mathcal{A}_{\Sigma})> 0$. To achieve a contradiction, we claim that it suffices to construct for each $j \in \{1, \dots, d\}$  a pair $\hat{X}, \hat{Y} \in \Pi_\Sigma$ (that is, $\hat{X}, \hat{Y}: \mathbb{R}^d \to \Sigma$ measurable maps such that $|x - \hat{X}(x)| = |x - \hat{Y}(x)| = \dist(x, \Sigma)$ for all $x \in \mathbb{R}^d$), such that
  \begin{align*}
    \hbox{$\hat{X}^j(x) = \max \pi_j(\mathcal{P}_{\Sigma}(x))$, and $\hat{Y}^j(x) = \min \pi_j(\mathcal{P}_{\Sigma}(x))$,}
  \end{align*}
  where we write $\hat{X}^j(x) = \pi_j(\hat{X}(x))$, and $\pi_j:
  \mathbb{R}^d \to \mathbb{R}, (x^1, \dots, x^d) \mapsto x^j$ is the
  projection map.  Indeed, suppose such a pair $\hat{X},
  \hat{Y}$ exist for each $j$. Notice
  that $$\mathcal{A}_{\Sigma}\subseteq \bigcup_{j =1}^d \{x \in
  \mathbb{R}^d \ | \ \max \pi_j(\mathcal{P}_{\Sigma}(x)) > \min
  \pi_j(\mathcal{P}_{\Sigma}(x))\}.$$ So, since
  $\mu(\mathcal{A}_{\Sigma}) > 0$, there is some $j \in \{1, \dots,
  d\}$ such that $$\{x \in \mathbb{R}^d \ |\ \max
  \pi_j(\mathcal{P}_{\Sigma}(x)) > \min
  \pi_j(\mathcal{P}_{\Sigma}(x))\} \cap \mathcal{A}_{\Sigma}$$ has positive measure with respect
  to $\mu$.

  By \cref{lemma:netbarycentrefieldiszero}, since $\mathscr{J}(\Sigma + a) \geq \mathscr{J}(\Sigma)$ for all $a$ in a neighbourhood of $0$, $\mathcal{B}_{\hat{X}}^{\mathrm{net}} = \mathcal{B}_{\hat{Y}}^{\mathrm{net}} = 0$. So, we have
  \begin{align*}
    0 &= \mathcal{B}_{\hat{X}}^{\mathrm{net}} - \mathcal{B}_{\hat{Y}}^{\mathrm{net}} \\
    &= \int_{\mathbb{R}^d}\frac{(x - \hat{X}(x))}{|x - \hat{X}(x)|}\phi'(|x - \hat{X}(x)|)d\mu(x) \\
    & \qquad - \int_{\mathbb{R}^d}\frac{(x - \hat{Y}(x))}{|x - \hat{Y}(x)|}\phi'(|x - \hat{Y}(x)|)d\mu(x)\\
    &= \int_{\mathcal{A}_{\Sigma}}(\hat{X}(x) - \hat{Y}(x))\frac{\phi'(\dist(x, \Sigma))}{\dist(x, \Sigma)}d\mu(x).
  \end{align*}
    For each $x \in \Sigma$, $\mathcal{P}_{\Sigma}(x) = \{x\}$, thus we have that $\dist(x, \Sigma) > 0$ for every $x \in \mathcal{A}_{\Sigma}$. So, since $\phi \in C^1([0, \infty))$ is strictly increasing, $\phi'(\dist(x, \Sigma)) > 0$ for all $x \in \mathcal{A}_{\Sigma}$. Moreover, we have $\hat{X}^j(x) - \hat{Y}^j(x) \geq 0$ for all $x \in \mathbb{R}^d$ by definition, and the inequality is strict on a positive measure subset of $\mathcal{A}_{\Sigma}$. So,
    \begin{align*}
        0 &= \pi_{j}\left(\int_{\mathcal{A}_{\Sigma}}(\hat{X}(x) - \hat{Y}(x))\frac{\phi'(\dist(x, \Sigma))}{\dist(x, \Sigma)}d\mu(x)\right) \\ &= \int_{\mathcal{A}_{\Sigma}}(\hat{X}^j(x) - \hat{Y}^j(x))\frac{\phi'(\dist(x, \Sigma))}{\dist(x, \Sigma)}d\mu(x) > 0,
    \end{align*}
    a contradiction.
    
  So, it remains to show that such $\hat{X}$ and $\hat{Y}$ exist. We show that $\hat{X}$ exists, the same argument applies to $\hat{Y}$ \textit{mutatis mutandis}. Fix $j \in \{1, \dots, d\}$, and consider $D = \{(x, \sigma) \in \mathbb{R}^d\times \Sigma \ | \ \sigma \in \mathcal{P}_{\Sigma}(x)\}$, then $D$ is closed. Let $f : D \to \mathbb{R}, (x, \sigma)\mapsto -\pi_j(\sigma)$, then $f$ is continuous. So, by \cref{lemma:measurableselection}, there exists a Borel-measurable function $\hat{X}: \mathbb{R}^d \to \Sigma$ such that $(x, \hat{X}(x)) \in \mathcal{P}_{\Sigma}(x)$ for each $x \in \mathbb{R}^d$, and for all $x \in \mathbb{R}^d$,
  \[
    f(x, \hat{X}(x)) = \pi_j(\hat{X}(x)) = \hat{X}^j(x) = \max_{\sigma \in \mathcal{P}_{\Sigma}(x)}\pi_j(\sigma).
  \]
  Thus, the desired $\hat{X} \in \Pi_{\Sigma}$ exists. 
\end{proof}

In particular, this allows us to conclude that any two closest-point projections onto an average distance minimizer are unique $\mu$-a.e.

\begin{corollary}[Uniqueness of closest-point projection]\label{cor:uniquenessofclosestpointprojections}
    Assume $\phi \in C^1([0, \infty))$ is strictly increasing. Let $\Sigma \in \mathcal{S}_{\ell}$ be an average distance minimizer. Then, $\mu(\mathcal{A}_{\Sigma}) = 0$, and for any two closest-point projections $\pi_{\Sigma}, \pi_{\Sigma}' \in \Pi_{\Sigma}$, 
    \begin{enumerate}
        \item $\pi_{\Sigma}(x) = \pi_{\Sigma}'(x)$ for $\mu$-a.e. $x \in \mathbb{R}^d$,
        \item $\mathcal{B}_{\pi_{\Sigma}} = \mathcal{B}_{\pi_{\Sigma}'}$, and 
        \item $\nu_{\pi_{\Sigma}} = \nu_{\pi_{\Sigma}'}$. 
    \end{enumerate}
    In particular, under these assumptions, we have 
    \[
    \nu_{\pi_{\Sigma}}(\{\sigma\}) = \mu(\{x \in \mathbb{R}^d \ | \ d(x, \sigma) = \dist(x, \Sigma)\})
    \]
    for every $\sigma \in \Sigma$.
\end{corollary}
\begin{proof}
    The set $\mathcal{S}_{\ell}$ is translation invariant, since for any $\Sigma \in \mathcal{S}_{\ell}$ and $a \in \mathbb{R}^d$, the set $a + \Sigma$ is compact, connected, and 
    \[
    \mathcal{H}^1(\Sigma + a) = \mathcal{H}^1(\Sigma) \leq \ell.
    \]
    So, \cref{prop:negligibilityofambiguouslocus} shows that $\mu(\mathcal{A}_{\Sigma}) = 0$ for any minimizer $\Sigma$ of the average distance problem. Since 
    \[
    \{x \in \mathbb{R}^d \ | \ \pi_{\Sigma}(x) \ne \pi_{\Sigma}'(x)\} \subseteq \mathcal{A}_{\Sigma},
    \]
    we conclude that $\pi_{\Sigma}(x) = \pi_{\Sigma}'(x)$ for $\mu$-a.e. $x$; this implies that $\mathcal{B}_{\pi_{\Sigma}} = \mathcal{B}_{\pi_{\Sigma}'}$ and $\nu_{\pi_{\Sigma}} = \nu_{\pi_{\Sigma}'}$. Moreover,  we have 
    \begin{align*}
        \{x \in \mathbb{R}^d \ | \ d(x, \sigma) = \dist(x, \Sigma)\} \setminus \mathcal{A}_{\Sigma} & \subseteq \pi_{\Sigma}^{-1}\{\sigma\} \\
        & \subset \{x \in \mathbb{R}^d \ | \ d(x, \sigma) = \dist(x, \Sigma)\},
    \end{align*}
    so we conclude that 
    \[
    \nu_{\pi_{\Sigma}}(\{\sigma\}) = \mu(\{x \in \mathbb{R}^d \ | \ d(x, \sigma) = \dist(x, \Sigma)\}).
    \]
\end{proof}

\subsection{Bounding higher-order terms}\label{sec:boundinghigherorderterms}

We now will prove a bound on the higher-order terms in the expansion given in the general first-order approximation \cref{prop:1generalfirstorderapproximationtheorem1}. This result will be particularly useful when bounding the change in $\mathscr{J}(\Sigma)$ under local modifications to $\Sigma$; such arguments will play a key role in the following sections, especially in the proof of existence of an atom/barycentre nontriviality for average distance minimizers \cref{theorem:existenceofatom}.

The parameters in the following bound play an important role in the proof of \cref{theorem:existenceofatom}. In particular, the ability to choose the value of the parameter $\delta$ appearing in \cref{prop:boundinghigherorderestimatescase1}, possibly depending on $\epsilon$, plays a pivotal role in this proof, as does the ability to choose (by carefully constructing $F$ and $G$) the set $A_{F,G}$. For the sake of concision, we now introduce some notation which we will use throughout the rest of the paper. Define the measure
\begin{equation}\label{eq:gammadefinition}
    d\gamma_F := \phi'(|x - F(x)|)d\mu;
\end{equation}
this will allow us to use the shorthand 
\[
\gamma_F(E) = \int_{E}\phi'(|x - F(x)|)d\mu(x)
\]
for a Borel set $E \subseteq \mathbb{R}^d$. When $F = \pi_{\Sigma}$ and $\Sigma$ is clear from context, we will often write $\gamma = \gamma_{\pi_{\Sigma}}$. Moreover, define the measure
\begin{equation}\label{eq:rhodefinition}
    \rho = \rho_{F} = F_{\#}(\gamma_F).
\end{equation}
Additionally, we denote
\begin{equation}\label{eq:mphinotation}
    M_{\phi}:= \max_{t \in [0, \diam( \supp \mu)]}\phi'(t);
\end{equation}
when $\phi$ is clear from context we simply write $M = M_{\phi}$. We denote the $\delta$-neighbourhood of a set $E \subseteq \mathbb{R}^d$ by
\[
B_{\delta}(E) = \bigcup_{\sigma \in E}B_{\delta}(\sigma).
\]
Finally, given $F \in \mathcal{F}$, we denote
\[
\mathcal{I}(F) := \{x \in \mathbf{R}^d \ | \ |F(x) - x| = 0\}. 
\]

\begin{proposition}[Bounding higher-order terms]\label{prop:boundinghigherorderestimatescase1}
    Assume $\phi \in C^{1,1}_{\mathrm{loc}}([0, \infty))$, and let $L$ be a Lipschitz constant for $\phi'$ on $[0, \mathrm{diam}(\supp \mu)]$. Let $F, G \in \mathcal{F}$, and denote $\epsilon = ||F - G||_{\infty}$. Suppose that $\mu(\mathcal{I}(F)) = \mu(\mathcal{I}(G)) = 0$. Denote 
    \begin{equation}\label{eq:afgnotation}
        A_{F,G} = \{x \in \mathbb{R}^d \ | \ F(x) \ne G(x)\}.
    \end{equation}
    Then, for each $\delta > 0$, 
    \begin{align*}
        &\mathscr{J}(F) - \mathscr{J}(G) -\int_{\mathbb{R}^d} (G(x) - F(x))\cdot \Delta_F(x)d\mu(x) \\
        & \qquad \geq - \frac{\epsilon^2}{2\delta}\gamma_F(A_{F,G} \setminus B_{\delta}(\mathcal{I}(F)))\\
        &\qquad \qquad - L\epsilon^2\mu(A_{F,G} \setminus B_{\delta}(\mathcal{I}(F)))-2M\epsilon\mu(B_{\delta}(\mathcal{I}(F)) \cap A_{F,G}).
    \end{align*}
\end{proposition}

\begin{proof}
    In what follows, we write $A = A_{F,G}$. Let $x \in \mathbb{R}^d \setminus (\mathcal{I}(F) \cup \mathcal{I}(G))$, then we may apply the mean value theorem to find some $\min\{|x - F(x)| , | x- G(x)|\} \leq c^x \leq \max\{ |x - F(x)| , |x - G(x)|\}$ such that 
     \[
    \phi(|x - F(x)|) - \phi(|x - G(x)|) = (|x - F(x)| - |x -G(x)|)\phi'(c^x).
    \]
    By the definition of $L$, we then have
    \begin{align*}
        &\phi(|x - F(x)|) - \phi(|x - G(x)|) \\
        & \qquad \geq (|x - F(x)| - |x -G(x)|)\phi'(|x - F(x)|) -L\epsilon^2.
    \end{align*}
    Using the inequality $\sqrt{1 + x} \leq 1 + \frac{1}{2}x$, we have
    \begin{equation}\label{eq:squarerootinequalityresult}
        \begin{split}
            &|x - G(x)| = |(x - F(x)) - (G(x) - F(x))|\\
            &\qquad = |x - F(x)|\sqrt{1 - 2(G(x) - F(x))\cdot\frac{(x - F(x))}{|x - F(x)|^2} + \frac{|G(x) - F(x)|^2}{|x - F(x)|^2}}\\
        &\qquad \leq |x - F(x)| - (G(x) - F(x))\cdot \frac{(x - F(x))}{|x - F(x)|} + \frac{1}{2}\frac{|G(x) - F(x)|^2}{|x -F(x)|}.
        \end{split}
    \end{equation}
    So, since $\mu(\mathcal{I}(F)) = \mu(\mathcal{I}(G)) = 0$, we find that for $\mu$-a.e. $x \in \mathbb{R}^d$ that
    \begin{equation*}
        \begin{split}
            &\phi(|x - F(x)|) - \phi(|x - G(x)|) + (G(x) - F(x))\cdot \Delta_F(x) \\
            & \qquad \geq -L\epsilon^2
        - \frac{1}{2}\frac{|G(x) - F(x)|^2}{|x - F(x)|}\phi'(|x - F(x)|).
        \end{split}
    \end{equation*}
    For $\mu$-a.e. $x \in A \setminus B_{\delta}(\mathcal{I}(F))$, we thus get that
    \begin{equation}\label{eq:lowerboundinAminusBdelta}
        \begin{split}
            &\phi(|x - F(x)|) - \phi(|x - G(x)|) + (G(x) - F(x))\cdot \Delta_F(x) \\
            &\qquad \geq -(L + \frac{1}{2\delta}\phi'(|x - F(x)|))\epsilon^2.
        \end{split}
    \end{equation}
    Meanwhile, if $x \in A \cap B_{\delta}(\mathcal{I}(F))$, then using that
    \[
    \phi(|x - F(x)|) - \phi(|x - G(x)|) \geq - \epsilon M
    \]
    and 
    \[
    |\Delta_F| \leq | \phi'(|x - F(x)|)| \leq M
    \]
    for $x \in \mathbb{R}^d \setminus \mathcal{I}(F)$,
    we may directly get the bound \begin{equation}\label{eq:lowerboundinAcapBdelta}
        \phi(|x - F(x)|) - \phi(|x - G(x)|) + (G(x) - F(x))\cdot \Delta_F(x) \geq -2\epsilon M
    \end{equation}
    for $\mu$-a.e. $x \in  A \cap B_{\delta}(\mathcal{I}(F))$. Finally, since
    \[
    \phi(|x - F(x)|) - \phi(|x - G(x)|) + (G(x) - F(x))\cdot \Delta_F(x) = 0
    \]
    for $x \in \mathbb{R}^d \setminus A$, integrating \eqref{eq:lowerboundinAcapBdelta} and \eqref{eq:lowerboundinAminusBdelta} over the regions in which they are given and adding the resulting bounds yields
     \begin{align*}
        &\mathscr{J}(F) - \mathscr{J}(G) -\int_{\mathbb{R}^d} (G(x) - F(x))\cdot \Delta_F(x)d\mu(x) \\
        & \qquad \geq - \frac{\epsilon^2}{2\delta}\gamma_F(A \setminus B_{\delta}(\mathcal{I}(F)))
        - L\epsilon^2\mu(A \setminus B_{\delta}(\mathcal{I}(F)))\\
        & \qquad \qquad -2M\epsilon\mu(B_{\delta}(\mathcal{I}(F)) \cap A),
    \end{align*}
    as desired. 
\end{proof}

\section{Bounding the mass of noncut points}\label{sec:boundingmassofnoncutpoints}

A key concept in the remainder of the paper is that of the \textit{noncut points} of an average distance minimizer $\Sigma \in \mathcal{S}_{\ell}$, that is, the points $x \in \Sigma$ such that $\Sigma \setminus \{x\}$ is connected. The importance of noncut points in the average distance problem comes from the idea that the $\nu$-mass of neighbourhoods of noncut points should be closely related to how much one can improve the objective value given $\epsilon$ additional budget. This suggests a natural link between the mass of noncut points and the barycentre field, which will be established in \cref{cor:nontrivialbarycentrefieldandatoms}, a key tool in the proof of existence of an atom \cref{theorem:existenceofatom}

The focus of this section will be on studying how the mass of noncut points can be bounded from below using improvements to the objective value. In \cref{sec:foundationalresults} we will introduce an important idea justifying why we should expect such bounds, before using this idea to prove in \cref{lemma:noncutneighbourhoodshavepositivemass} that $\epsilon$-neighbourhoods of noncut points have $\nu$-mass bounded below by a constant times $\epsilon$ as $\epsilon \to 0$. While \cref{lemma:noncutneighbourhoodshavepositivemass} is certainly not sharp, it will later be used in \cref{theorem:existenceofatom} to bootstrap to existence of an atom. In \cref{sec:thebarycentrefieldandatomicnoncutpoints}, we will prove the main theorem of this section, \cref{theorem:boundingmassofnoncutpoints}, which provides a lower bound on the $\nu$-mass of noncut points in terms of the barycentre field. This generalizes \cite{Obrie25}*{Theorem 3.2} to a much greater range of functions $\phi$, in particular including $\phi(t) = t^p$ for any $p \geq 1$. Using \cref{theorem:boundingmassofnoncutpoints}, we will then prove the existence of an atom for minimizers of the \textit{soft-penalty} average distance problem in \cref{sec:softpenatlyatom}, following the arguments of \cite{Obrie25}*{Section 1.1.2}.

We begin in \cref{sec:foundationalresults} with a discussion of a critical idea from \cite{Buttazzo03}, which will form the basis for our proofs of \cref{lemma:noncutneighbourhoodshavepositivemass} and \cref{theorem:boundingmassofnoncutpoints}.

\subsection{A critical idea}\label{sec:foundationalresults}

Given a connected topological space $X$, we say a point $p \in X$ is a \textit{noncut point} if $X \setminus \{p\}$ is connected. For example, the endpoints $0$ and $1$ of the line segment $[0,1]$ are the only noncut points of this space, while every point of the unit circle $p \in \mathbb{S}^1$ is a noncut point. We remark that every $\Sigma \in \mathcal{S}_{\ell}$ containing at least two points will contain at least two noncut points by \cite{kuratowski}*{§47 Theorem IV.5}.

The first key relationship between the $\nu$-mass of noncut points and improvements to the objective value comes from the following critical idea used by Buttazzo and Stepanov in \cite{Buttazzo03}*{Proposition 7.1}. Given a noncut point $\sigma$ in a minimizer $\Sigma \in \mathcal{S}_{\ell}$, we expect to find a “noncut-neighbourhood” $B_{\epsilon} \subseteq \Sigma$ of $\sigma$ with $\mathcal{H}^1(B_{\epsilon}) = \epsilon$ such that $\Sigma \setminus B_{\epsilon}$ is connected, and thus is in $\mathcal{S}_{\ell-\epsilon}$. Removing $B_{\epsilon}$ from $\Sigma$ should only increase the objective value by at most a constant times $\epsilon\nu(B_{\epsilon})$. So, if we can use the $\epsilon$ additional budget to improve the objective value of $\Sigma \setminus B_{\epsilon}$ by $f(\epsilon)$, then to avoid contradicting the minimality of $\Sigma$ we should have
\[
\nu(B_{\epsilon}) \geq C\frac{f(\epsilon)}{\epsilon}
\]
for some $C$. 

In particular, if we assumed that the barycentre field was nontrivial, then we would be able to decrease the objective value by $C'\epsilon$ for some constant $C' > 0$ by \cref{prop:lipschitzapproximation}. This would give us a lower bound on $\nu(B_{\epsilon})$ that is uniform in $\epsilon$; in particular, by taking $\epsilon$ to $0$, we could conclude that every noncut point must be an atom. This particular argument is made rigorous in \cref{theorem:boundingmassofnoncutpoints}.

In \cref{lemma:noncutneighbourhoodshavepositivemass}, we combine the above idea with \cite{Stepanov04}*{Lemma 3.6}, which says that one can always decrease the objective value by the order $\epsilon^2$ given $\epsilon$ additional budget, to prove that any $\epsilon$ neighbourhood of a noncut point has $\nu$-mass bounded below by $C \epsilon$ for some constant $C$. We begin by stating \cite{Buttazzo03}*{Lemma 6.1}, which provides the existence of the “noncut-neighbourhoods” described in the above sketch. 
\begin{lemma}[Noncut-neighbourhood lemma, \cite{Buttazzo03}*{Lemma 6.1}]\label{lemma:noncutneighbourhood}
  Let $\Sigma$ be a locally connected metric continuum containing more than one point, and let $\sigma \in \Sigma$ be a noncut point of $\Sigma$. Then, there exists a sequence $\{B_n\}_{n \in \mathbb{N}}$ of open subsets of $\Sigma$ satisfying the following conditions:
  \begin{itemize}
    \item[(i)] $\sigma \in B_n$ for all sufficiently large $n$,
    \item[(ii)] $\Sigma \setminus B_n$ is connected for each $n \in \mathbb{N}$,
    \item[(iii)] $\diam(B_n) \to 0$ as $n \to \infty$, and
    \item[(iv)] $B_n$ is connected for every $n$.
  \end{itemize}
    We will refer to the sets $\{B_n\}_{n \in \mathbb{N}}$ as a \textit{noncut-neighbourhood system} for $\sigma$.
\end{lemma}

Given a minimizer $\Sigma \in \mathcal{S}_{\ell}$ and $\pi_{\Sigma}  \in \Pi_{\Sigma}$, define the measure $\rho$ by 
\begin{equation}
    \rho = \rho_{\pi_{\Sigma}} := (\pi_{\Sigma})_{\#}\gamma,
\end{equation}
where we recall the definition of $\gamma$ from \eqref{eq:gammadefinition}. The measure $\rho$ is be helpful to consider due to the fact that if $\phi \in C^1([0, \infty))$ is nondecreasing, then by the triangle inequality we have
\[
\int_{E}|\mathcal{B}_{\pi_{\Sigma}}(\sigma)|d\nu(\sigma) \leq \int_{\pi_{\Sigma}^{-1}(E)}\phi'(|x - \pi_{\Sigma}(x)|) d\mu(x) = \gamma(\pi_{\Sigma}^{-1}(E)) = \rho(E)
\]
for any Borel $E \subseteq \mathbb{R}^d$.
In the proof of \cref{theorem:existenceofatom}, it is necessary to know that the $\rho$-mass of noncut-neighbourhoods can be bounded below by $C\epsilon$, so we will prove \cref{lemma:noncutneighbourhoodshavepositivemass} in terms of $\rho$. However, this easily implies that noncut neighbourhoods have $\nu$-mass bounded below by $C\epsilon$, thanks to the inequality
\[
\rho(E) \leq M\nu(E)
\]
for any Borel $E \subseteq \mathbb{R}^d$, where $M$ is defined in \eqref{eq:mphinotation}.

\begin{proposition}[$\epsilon$-neighbourhoods of noncut points have $\nu$-mass bounded below by $\epsilon$]\label{lemma:noncutneighbourhoodshavepositivemass}
    Assume $\phi \in C^{1,1 }_{\mathrm{loc}}([0, \infty))$ satisfies \eqref{eq:alpha2}, and suppose $\mu(B_{\epsilon}(x)) = o(\epsilon)$ for every $x \in \mathbb{R}^d$. Let $\Sigma \in \mathcal{S}_{\ell}$ be optimal, and let $\sigma^* \in \Sigma$ be a noncut point. Let $\{B_n\}_{n \in \mathbb{N}}$ be a noncut-neighbourhood system for $\sigma^*$, and let $\epsilon_n := \dist(\sigma^*, \partial B_n)$. Then,
    \[
    \lim_{n \to \infty}\frac{\rho(B_n)}{\epsilon_n} > 0,
    \]
    and so in particular 
    \[
    \lim_{n \to \infty}\frac{\nu(B_n)}{\epsilon_n} > 0.
    \]
    \end{proposition}

    \begin{proof}
        For the sake of contradiction, assume that $\nu(B_n) = o(\epsilon_n)$. For each $n \in \mathbb{N}$, define
        \[
        \Sigma_n := \Sigma \setminus B_n,
        \]
        then $\Sigma_n$ is compact, connected, and $\mathcal{H}^1(\Sigma_n)\leq \ell - \epsilon_n$. We now wish to apply \cite{Stepanov04}*{Lemma 3.6} to find some $\Sigma'_n$ which is compact, connected, $\mathcal{H}^1(\Sigma'_n)\leq \ell$, and 
        \begin{equation}\label{eq:paoliniandstepanovimpliesthis}
            \mathscr{J}(\Sigma'_n) \leq \mathscr{J}(\Sigma_n) - C\epsilon_n^2
        \end{equation}
        for some $C > 0$ independent of $n$. Firstly, notice that we are assuming Paolini and Stepanov's condition \eqref{eq:alpha2}. Moreover, since $\mu(\Sigma) = 0$, we may find some $\eta > 0$ so that $\mu(\mathbb{R}^d \setminus B_{\eta}(\Sigma)) > 0$; taking $H = B_{\eta/2}(\Sigma)$ and $K = \mathbb{R}^d \setminus B_{\eta}(\Sigma)$, we see that 
        \[\inf\{\dist(y, H) \ | \ y \in K\} = \frac{\eta}{2} > 0\]
        and $\mu(K) > 0$. So, we see that the assumptions of \cite{Stepanov04}*{Lemma 3.6} are satisfied, thus there is some constant $C > 0$ independent of $n$ so that for all sufficiently large $n$, we may find some competitor $\Sigma'_n \in \mathcal{S}_{\ell}$ satisfying \eqref{eq:paoliniandstepanovimpliesthis}. 

    Now, we use \cref{prop:boundinghigherorderestimatescase1} to bound the difference in objective value between $\Sigma$ and $\Sigma_n$. Let  
    \[
        G_n(x) = \bigg\{ \begin{matrix}
            \pi_{\partial B_n}( \pi_{\Sigma}(x)), & x \in \pi_{\Sigma}^{-1}(B_n), \\
            \pi_{\Sigma}(x), & \text{ otherwise},
        \end{matrix}
        \]
    and let $F = \pi_{\Sigma}$. Notice that $A_{F, G_n} = \pi_{\Sigma}^{-1}(B_n)$, so combining the bound $\gamma_{F} \leq M\mu$ (recall the definition of $M$ from \eqref{eq:mphinotation}) with \cref{prop:boundinghigherorderestimatescase1} and taking $\delta =1$, we have
    \begin{align*}
        \mathscr{J}(\Sigma) - \mathscr{J}(\Sigma_n) &\geq \mathscr{J}(\pi_{\Sigma}) - \mathscr{J}(G_n) \\
        &= \int_{B_n}(\sigma - \pi_{\partial B_n}(\sigma))\cdot \mathcal{B}_{\pi_{\Sigma}}(\sigma)d\nu(\sigma) \\ & \qquad - (\frac{1}{2}M + L)\epsilon_n^2 \nu(B_n) - 2M \epsilon_n \nu(B_n).
    \end{align*}
    So, since $\nu(B_n) = o(\epsilon_n)$ by assumption, we have
    \[
    \mathscr{J}(\Sigma) - \mathscr{J}(\Sigma_n) \geq \int_{B_n}(\sigma - \pi_{\partial B_n}(\sigma))\cdot \mathcal{B}_{\pi_{\Sigma}}(\sigma)d\nu(\sigma) + o(\epsilon_n^2).
    \]
    Now, using the fact that $\Sigma$ is optimal, we have that for each $n$,
    \begin{align*}
        0 &\geq \mathscr{J}(\Sigma) - \mathscr{J}(\Sigma_n') \\&= (\mathscr{J}(\Sigma)- \mathscr{J}(\Sigma_n)) + (\mathscr{J}(\Sigma_n) - \mathscr{J}(\Sigma_n')) \\
        &\geq \int_{B_n}(\sigma - \pi_{\partial B_n}(\sigma))\cdot \mathcal{B}_{\pi_{\Sigma}}(\sigma)d\nu(\sigma) + C\epsilon_n^2  + o(\epsilon_n^2).
    \end{align*}
    Thus, we see that 
    \[
    \lim_{n \to \infty}\frac{1}{\epsilon_n^2}\int_{B_n}( \pi_{\partial B_n}(\sigma) - \sigma)\cdot \mathcal{B}_{\pi_{\Sigma}}(\sigma)d\nu(\sigma) \geq C > 0.
    \]
    But we know that 
    \[
    \frac{1}{\epsilon_n^2}\int_{B_n}( \pi_{\partial B_n}(\sigma) - \sigma)\cdot \mathcal{B}_{\pi_{\Sigma}}(\sigma)d\nu(\sigma) \leq \frac{\rho(B_n)}{\epsilon_n} \leq  M \frac{\nu(B_n)}{\epsilon_n},
    \]
    where we recall the definition of $M$ from \eqref{eq:mphinotation}. So, we conclude that
    \[
    \lim_{n \to \infty}\frac{\nu(B_n)}{\epsilon_n} \geq \lim_{n \to \infty}\frac{1}{M}\frac{\rho(B_n)}{\epsilon_n} > 0,
    \]
    a contradiction. 
\end{proof}

\subsection{The barycentre field and atomic noncut points}\label{sec:thebarycentrefieldandatomicnoncutpoints}

Now, we will proceed with \cref{theorem:boundingmassofnoncutpoints}, the main result of this section. Our proof will again use the argument outlined in \cref{sec:foundationalresults}, this time using the barycentre field to improve the objective value. 

\begin{theorem}[Bounding the mass of noncut points]\label{theorem:boundingmassofnoncutpoints}
Assume $\phi \in C^1([0, \infty))$ is nondecreasing, and assume $\mu(\Sigma) = 0$ for every $\Sigma \in \mathcal{S}$. Suppose $\Sigma \in
  \mathcal{S}_{\ell}$ is an optimizer and $\Sigma$ contains at least two
  points. Let $\pi_\Sigma \in \Pi_\Sigma$ and $\nu =
  (\pi_{\Sigma})_{\#}\mu$. Let $\sigma^* \in \Sigma$ be a noncut
  point. Then, 
  \begin{equation}\label{eq:lowerboundonatomfrombarycentrefield}
    \sup_{\zeta \in \mathrm{Lip}^*(\Sigma)}\int_{\Sigma} \zeta(\sigma)\cdot \mathcal{B}_{\pi_{\Sigma}}(\sigma)d\nu_{\pi_{\Sigma}}(\sigma)\leq \ell|\mathcal{B}_{\pi_{\Sigma}}(\sigma^*)|\nu_{\pi_{\Sigma}}(\sigma^*),
  \end{equation}
  where $\mathrm{Lip}^*(\zeta) = \{\zeta: \Sigma \to \mathbb{R}^d \ | \ \zeta \text{ is 1-Lipschitz}\}$.
\end{theorem}

\begin{proof}
    If $\mathcal{B}_{\pi_{\Sigma}}$ is trivial then the claim clearly holds, so assume $\mathcal{B}_{\pi_{\Sigma}}$ is nontrivial. Our plan is to remove a noncut-neighbourhood of $\sigma^*$ from $\Sigma$, gaining back $\epsilon$ budget, and then use the barycentre field to construct a competitor to $\Sigma$ with this $\epsilon$ extra length. 
    
    By \cref{lemma:noncutneighbourhood} $\{B_n\}_{n \in \mathbb{N}}$ be a noncut-neighbourhood system for $\sigma^*$, and let $\epsilon_n := \diam (B_n)$. By passing to a subsequence (not relabelled), we may assume that $\sigma^* \in B_n$ for all $n \in \mathbb{N}$. Moreover, since $\epsilon_n \to 0$ as $n \to \infty$, we may assume that $\epsilon_n < \ell $ for all $n \in \mathbb{N}$. Let $\zeta \in \mathrm{Lip}^*(\Sigma)$, and for each $n \in \mathbb{N}$ define the competitor 
    \[
    \Sigma_n := \{\sigma + \frac{\epsilon_n}{\ell - \epsilon_n}\zeta(\sigma) \ |\ \sigma \in \Sigma \setminus B_n \}.
    \]
    The set $\Sigma \setminus B_n$ is compact and connected by our choice of $B_n$, and $\sigma \mapsto \sigma + \frac{\epsilon_n}{\ell - \epsilon_n}\zeta(\sigma)$ is continuous, so we see that $\Sigma_n$ is compact and connected. Moreover, since $\zeta$ is $1$-Lipschitz, the map $\sigma \mapsto \sigma + \frac{\epsilon_n}{\ell - \epsilon_n}\zeta(\sigma)$ is $1 + \frac{\epsilon_n}{\ell - \epsilon_n}$-Lipschitz, and thus
    \begin{align*}
        \mathcal{H}^1(\Sigma_n) &\leq (1 + \frac{\epsilon_n}{\ell - \epsilon_n})\mathcal{H}^1(\Sigma \setminus B_n) \\
        &\leq\ell- \epsilon_n + \epsilon_n \\
        &= \ell,
    \end{align*}
    so $\Sigma_n \in \mathcal{S}_{\ell}$ for each $n \in \mathbb{N}$.

    Now that we have constructed our competitor $\Sigma_n$, we will approximate $\mathscr{J}(\Sigma_n)$. Fix $\delta > 0$, and define $C_{\delta} := B_{\delta}(\sigma^*) \cap \Sigma$. Then, since $\epsilon_n \to 0$ as $n \to \infty$, we may pass to a further subsequence to assume that $B_n \subseteq C_{\delta}$ for all $n \in \mathbb{N}$. Denote $H_n(x) := \pi_{\Sigma \setminus B_n}(\pi_{\Sigma}(x))$, and define
    \[
F_n := \Big \{\begin{matrix}
    \pi_{\Sigma}(x) + \frac{\epsilon_n}{\ell - \epsilon_n} \zeta(\pi_{\Sigma}(x)), & x \in \pi_{\Sigma}^{-1}(\Sigma \setminus C_{\delta}), \\
   H_n(x) + \frac{\epsilon_n}{\ell - \epsilon_n} \zeta(H_n(x)), & x \in \pi_{\Sigma}^{-1}(C_{\delta}).
\end{matrix}
\]
We wish to approximate $\mathscr{J}(\Sigma)$ by $\mathscr{J}(F_n)$. Indeed, if $x \in \pi_{\Sigma}^{-1}(C_{\delta})$, then $H_n(x) \in \Sigma \setminus B_n$ by definition, whereas if $x \in \pi_{\Sigma}^{-1}(\Sigma \setminus C_{\delta})$, then since $B_n \subseteq C_{\delta}$, $\pi_{\Sigma}(x) \in \Sigma \setminus B_{n}$. In either case, we have $F_n \in \Sigma_n$, so $\img(F_n) \subseteq \Sigma_n$. Therefore, by \cref{remark:extensionofaveragedistancefunctional}, we know that $\mathscr{J}(F_n) \geq \mathscr{J}(\Sigma_n)$. Recalling the notation $\mathscr{J}|_{A}$ from \eqref{eq:restrictedadfnotation}, we have that
\begin{align*}
    \mathscr{J}(\Sigma) - \mathscr{J}(\Sigma_n) &\geq \mathscr{J}(\Sigma) - \mathscr{J}(F_n)\\ &= \mathscr{J}(\Sigma) - \mathscr{J}|_{\pi_{\Sigma}^{-1}(\Sigma \setminus C_{\delta})}(\pi_{\Sigma} + \frac{\epsilon_n}{\ell - \epsilon_n}\zeta\circ \pi_{\Sigma}) \\ & \qquad -  \mathscr{J}|_{\pi_{\Sigma}^{-1}(C_{\delta})}(H_n + \frac{\epsilon_n}{\ell - \epsilon_n}\zeta \circ H_n).
\end{align*}
Notice that 
\begin{align*}
    \mathscr{J}|_{\pi_{\Sigma}^{-1}(\Sigma \setminus C_{\delta})}(\pi_{\Sigma} + \frac{\epsilon_n}{\ell - \epsilon_n}\zeta\circ \pi_{\Sigma}) &= \mathscr{J}(\pi_{\Sigma} + \frac{\epsilon_n}{\ell - \epsilon_n}\zeta\circ \pi_{\Sigma}) \\
    &\qquad - \mathscr{J}|_{\pi_{\Sigma}^{-1}(C_{\delta})}(\pi_{\Sigma} + \frac{\epsilon_n}{\ell - \epsilon_n}\zeta\circ \pi_{\Sigma}),
\end{align*}
and so we can divide our lower bound on $\mathscr{J}(\Sigma) - \mathscr{J}(\Sigma_n)$ into three terms which are amenable to bounding using \cref{prop:1generalfirstorderapproximationtheorem1}:
\begin{align*}
    \mathscr{J}(\Sigma) - \mathscr{J}(\Sigma_n) &\geq \mathscr{J}(\Sigma) - \mathscr{J}(\pi_{\Sigma} + \frac{\epsilon_n}{\ell - \epsilon_n}\zeta\circ \pi_{\Sigma}) \\
    &\qquad + \mathscr{J}|_{\pi_{\Sigma}^{-1}(C_{\delta})}(\pi_{\Sigma} + \frac{\epsilon_n}{\ell - \epsilon_n}\zeta\circ \pi_{\Sigma}) - \mathscr{J}|_{\pi_{\Sigma}^{-1}(C_{\delta})}(\pi_{\Sigma}) \\
    &\qquad + \mathscr{J}|_{\pi_{\Sigma}^{-1}(C_{\delta})}(\pi_{\Sigma}) - \mathscr{J}_{\pi_{\Sigma}^{-1}(C_{\delta})}(H_n + \frac{\epsilon_n}{\ell - \epsilon_n}\zeta\circ H_n).
\end{align*}
For the first term, we apply \cref{prop:1generalfirstorderapproximationtheorem1} to see that 
\begin{equation}\label{eq:firstterminlowerboundonnoncutpoints}
    \liminf_{n \to \infty}\frac{\mathscr{J}(\Sigma) - \mathscr{J}(\pi_{\Sigma} + \frac{\epsilon_n}{\ell - \epsilon_n})}{\frac{\epsilon_n}{\ell - \epsilon_n}} \geq \int_{\Sigma}\zeta(\sigma)\cdot \mathcal{B}_{\pi_{\Sigma}}(\sigma)d\nu_{\pi_{\Sigma}}(\sigma).
\end{equation}
Notice that $\mathscr{J}|_{A} = \mathscr{J}^{\mu \llcorner A}$, so we have no issues applying \cref{prop:1generalfirstorderapproximationtheorem1} to the second term to see that
\begin{align*}
    &\limsup_{n \to \infty}\frac{\mathscr{J}|_{\pi_{\Sigma}^{-1}(C_{\delta})}(\Sigma) - \mathscr{J}|_{\pi_{\Sigma}^{-1}(C_{\delta})}(\pi_{\Sigma} + \frac{\epsilon_n}{\ell - \epsilon_n}\zeta \circ \pi_{\Sigma})}{\frac{\epsilon_n}{\ell - \epsilon_n}} \\
    & \qquad \leq \int_{C_{\delta}}\zeta(\sigma)\cdot \mathcal{B}_{\pi_{\Sigma}}(\sigma)d\nu_{\pi_{\Sigma}}(\sigma),
\end{align*}
and so in particular
\begin{equation}\label{eq:secondterminlowerboundonnoncutpoints}
\begin{split}
    &\liminf_{n \to \infty}\frac{ \mathscr{J}|_{\pi_{\Sigma}^{-1}(C_{\delta})}(\pi_{\Sigma} + \frac{\epsilon_n}{\ell - \epsilon_n}\zeta \circ \pi_{\Sigma}) - \mathscr{J}|_{\pi_{\Sigma}^{-1}(C_{\delta})}(\Sigma)}{\frac{\epsilon_n}{\ell - \epsilon_n}} \\& \qquad \geq -  \int_{C_{\delta}}\zeta(\sigma)\cdot \mathcal{B}_{\pi_{\Sigma}}(\sigma)d\nu_{\pi_{\Sigma}}(\sigma).
\end{split}
\end{equation}

Finally, we approximate the third term. Notice that for each $x \in \pi_{\Sigma}^{-1}(C_{\delta})$,
\[
|H_n(x)  - \pi_{\Sigma}(x)| \leq \epsilon_n = \frac{\epsilon_n}{\ell - \epsilon_n}(\ell - \epsilon_n).
\]
In particular, we have that 
\[
\lim_{n \to \infty}\zeta(H_n(x)) = \zeta(\pi_{\Sigma}(x)),
\]
and
\[
\lim_{n \to \infty}\frac{|H_n(x)  - \pi_{\Sigma}(x)|}{\frac{\epsilon_n}{\ell - \epsilon_n}} =\ell.
\]
So, applying \cref{prop:1generalfirstorderapproximationtheorem1}, we have
\begin{equation}\label{eq:thirdterminlowerboundonnoncutpoints}
   \begin{split}
        &\liminf_{n \to \infty}\frac{\mathscr{J}|_{\pi_{\Sigma}^{-1}(C_{\delta})}(\Sigma) - \mathscr{J}|_{\pi_{\Sigma}^{-1}(C_{\delta})}(H_n + \frac{\epsilon_n}{\ell - \epsilon_n}\zeta \circ H_n)}{\frac{\epsilon_n}{\ell - \epsilon_n}} \\
        & \qquad \geq - \ell\int_{C_{\delta}}|\mathcal{B}_{\pi_{\Sigma}}(\sigma)|d\nu_{\pi_{\Sigma}}(\sigma) + \int_{C_{\delta}}\zeta(\sigma)\cdot \mathcal{B}_{\pi_{\Sigma}}(\sigma)d\nu_{\pi_{\Sigma}}(\sigma).
   \end{split}
\end{equation}
Combining \eqref{eq:firstterminlowerboundonnoncutpoints}, \eqref{eq:secondterminlowerboundonnoncutpoints}, and \eqref{eq:thirdterminlowerboundonnoncutpoints}, we conclude that
\[
\liminf_{n \to \infty}\frac{\mathscr{J}(\Sigma)- \mathscr{J}(\Sigma_n)}{\frac{\epsilon_n}{\ell - \epsilon_n}} \geq \int_{\Sigma}\zeta(\sigma)\cdot \mathcal{B}_{\pi_{\Sigma}}(\sigma)d\nu_{\pi_{\Sigma}}(\sigma) -\ell \int_{C_{\delta}}|\mathcal{B}_{\pi_{\Sigma}}(\sigma)|d\nu_{\pi_{\Sigma}}(\sigma).
\]
To avoid contradicting the optimality of $\Sigma$, we must therefore have
\[
\int_{\Sigma}\zeta(\sigma)\cdot \mathcal{B}_{\pi_{\Sigma}}(\sigma)d\nu_{\pi_{\Sigma}}(\sigma)\leq \ell \int_{C_{\delta}}|\mathcal{B}_{\pi_{\Sigma}}(\sigma)|d\nu_{\pi_{\Sigma}}(\sigma)
\]
for all $\delta > 0$. Taking $\delta \to 0$, we thus conclude that
\[
\int_{\Sigma}\zeta(\sigma)\cdot \mathcal{B}_{\pi_{\Sigma}}(\sigma)d\nu_{\pi_{\Sigma}}(\sigma)\leq \ell|\mathcal{B}_{\pi_{\Sigma}}(\sigma^*)|\nu_{\pi_{\Sigma}}(\sigma^*),
\]
as desired. 
\end{proof}

\begin{corollary}\label{cor:boundingmassofnoncutpoints}
    Assume $\phi \in C^1([0, \infty))$ is nondecreasing, and let $\Sigma$ be a solution to the average distance problem. Then, there exists some constant $C > 0$ such that for every noncut point $\sigma^* \in \Sigma$,

    \begin{equation}\label{eq:corboundingmassofnoncutpoints}
        C \int_{\Sigma}|\mathcal{B}_{\pi_{\Sigma}}(\sigma)| d\nu_{\pi_{\Sigma}}(\sigma)\leq |\mathcal{B}_{\pi_{\Sigma}}(\sigma^*)|\nu_{\pi_{\Sigma}}(\sigma^*).
    \end{equation}
    In particular, if $\mathcal{B}_{\pi_{\Sigma}}$ is nontrivial, then every noncut point of $\Sigma$ is an atom, and $\Sigma$ has finitely many noncut points. 
\end{corollary}

\begin{proof}
    Lipschitz functions are dense in $L^2(\mu)$, so we may take $\xi$ to be a Lipschitz map so that
    \[
    \int_{\Sigma}\xi(\sigma)\cdot \mathcal{B}_{\pi_{\Sigma}}(\sigma) d\nu_{\pi_{\Sigma}}(\sigma) \geq \frac{1}{2}\int_{\Sigma} |\mathcal{B}_{\pi_{\Sigma}}(\sigma)| d\nu_{\pi_{\Sigma}}(\sigma).
    \]
    Taking $L$ to be a nonzero Lipschitz constant for $\xi$, we see that $\zeta := \frac{1}{L} \xi$ is 1-Lipschitz, and so by \cref{theorem:boundingmassofnoncutpoints}, for any noncut point $\sigma^* \in \Sigma$,
    \[
    \frac{1}{2L\ell}\int_{\Sigma}|\mathcal{B}_{\pi_{\Sigma}}(\sigma)|d\nu_{\pi_{\Sigma}} \leq \frac{1}{\ell}\int_{\Sigma}\zeta(\sigma)\cdot \mathcal{B}_{\pi_{\Sigma}}(\sigma)d\nu_{\pi_{\Sigma}}(\sigma) \leq |\mathcal{B}_{\pi_{\Sigma}}(\sigma^*)| \nu_{\pi_{\Sigma}}(\sigma^*).
    \]
    This proves that \eqref{eq:corboundingmassofnoncutpoints} holds.
    
    If $\mathcal{B}_{\pi_{\Sigma}}$ is nontrivial, then this inequality uniformly bounds $|\mathcal{B}_{\pi_{\Sigma}}(\sigma^*)|$ away from $0$ for noncut points $\sigma^*$, and in particular every noncut point of $\Sigma$ must be an atom. Moreover, since $|\mathcal{B}_{\pi_{\Sigma}}(\sigma)|$ can be bounded from above by $M_{\phi}$ for $\nu$-a.e. $\sigma$ (recall the notation $M_{\phi}$ from \eqref{eq:mphinotation}), we conclude that \eqref{eq:corboundingmassofnoncutpoints} provides a uniform lower bound on the $\nu$-mass of the noncut points $\sigma^*$ of $\Sigma$, and in particular, since $\nu(\Sigma) =1$, $\Sigma$ has only finitely many noncut points in this case.
\end{proof}

\begin{remark}\label{remark:noteveryatomisanoncutpoint}
Notice that while nontriviality of the barycentre field implies that every noncut point is an atom, it does not seem to be true that every atom is necessarily a noncut point: indeed, the blow-up analysis of Santambrogio and Tilli \cite{Santambrogio05} does not rule out the existence of atoms at corner points of minimizers, and Buttazzo, Manini, and Stepanov \cite{Buttazzo09}*{Proposition 4.1} have constructed an example $\Sigma$ which contains an atom at a corner point, and is stationary for the soft-penalty average distance functional \eqref{eq:softpenaltyaveragedistancefunctional}. 
\end{remark}

\subsection{Existence of an atom for soft-penalty minimizers}\label{sec:softpenatlyatom}

Following the arguments of \cite{Obrie25}*{Section 1.1.2}, we now prove existence of an atom for minimizers of the \textit{soft-penalty} average distance problem. 

For any $\lambda > 0$, define the soft-penalty average distance functional by 
\begin{equation}\label{eq:softpenaltyaveragedistancefunctional}
    (\mathscr{J}^{\mu}_{\phi})^{\lambda}(\Sigma) := \mathscr{J}^{\mu}_{\phi}(\Sigma) + \lambda \mathcal{H}^1(\Sigma),
\end{equation}
for $\Sigma \subseteq \mathbb{R}^d$ compact. Define 
\begin{equation}\label{eq:softpenaltyconstraintset}
    \mathcal{S} := \bigcup_{l \geq 0}\mathcal{S}_{\ell} = \{\Sigma \subseteq \mathbb{R}^d \ | \ \Sigma \text{ is compact, connected, and }\mathcal{H}^1(\Sigma) < \infty\},
\end{equation}
then the \textit{soft-penalty average distance problem} asks us to find $\Sigma_{\lambda} \in \mathcal{S}$ such that
\begin{equation}\label{eq:softpenaltyadp}    (\mathscr{J}_{\phi}^{\mu})^{\lambda}(\Sigma_{\lambda}) = \inf_{\Sigma \in \mathcal{S}}(\mathscr{J}_{\phi}^{\mu})^{\lambda}(\Sigma).
\end{equation}
\begin{remark}[Comparison of hard-constraint and soft-penalty average distance problems]\label{remark:comparisonofhardconstraintandsoftpenalty}
    Every minimizer $\Sigma_{\lambda} \in \mathcal{S}$ of the soft-penalty average distance problem solves the hard-constraint average distance problem with length budget $\ell = \mathcal{H}^1(\Sigma_{\lambda})$. However, it is \textit{not} known under what conditions, if any, a hard-constraint average distance minimizer will also solve a soft-penalty problem for some $\lambda > 0$, see \cite{Lemenant12}*{Remark 22}. 
    
    If $\mathcal{H}^1(\supp  \mu) < \infty$, then one can easily find hard-constraint minimizers which do not solve a soft-penalty problem. Indeed, for any compactly supported $\mu$ with $\mathcal{H}^1(\supp \mu) < \infty$, one can find a compact, connected set $\Sigma \subset \mathbb{R}^d$ with $\Sigma \supset \supp \mu$ and $\mathcal{H}^1(\Sigma) < \infty$. Such a $\Sigma$ will certainly minimize $\mathscr{J}^{\mu}$ over $\mathcal{S}_{\mathcal{H}^1(\Sigma)}$, but will not solve the soft-penalty problem for any $\lambda$ if there is some $\Sigma ' \in \mathcal{S}_{\ell}$ with $\Sigma' \supset \supp \mu$ and $\mathcal{H}^1(\Sigma')<\mathcal{H}^1(\Sigma)$.

    Surprisingly, we can also find examples of hard-constraint minimizers which also minimize $\mathcal{H}^1(\Sigma')$ over all $\Sigma' \in \mathcal{S}$ containing $\supp \mu$, but are not soft-penalty minimizers. To do so, take $d \geq 2$, $\phi(t) = t^p$ for $p = 1$ or $p \geq 2$, $\Sigma^+ = \bigcup_{i=1}^d \{te_i \ | \ t \in [-\frac{1}{2d}, \frac{1}{2d}]\}$, and let $\mu = \mathcal{H}^1 \llcorner \Sigma^+$ be the uniform measure on $\Sigma^+$. Clearly, $\Sigma^+$ minimizes the average distance functional over $\mathcal{S}_1$, and moreover $\Sigma^+$ minimizes $\mathcal{H}^1(\Sigma^+)$ over all $\Sigma' \in \mathcal{S}$ containing $\supp \mu$. Despite this, $\Sigma^+$ cannot solve the soft-penalty average distance problem for any $\lambda > 0$. Indeed, since $\mu$ is compactly supported, any minimizer of a soft-penalty average distance problem cannot contain a point of order (see \eqref{eq:orddefinition}) greater than $3$; this is guaranteed by \cite{Slepcev13}*{Lemma 3.4} in the case that $\phi(t) = t$, and \cite{Obrie25}*{Corollary 1.5} in the case that $\phi(t) = t^p$ for $p \geq 2$. The point $0 \in \Sigma^+$ has order $2d \geq 4$, showing it is not a soft-penalty minimizer for any $\lambda$.

    This example helps motivate why proving the topological characterization for the hard-constraint problem is more difficult than for the soft-penalty problem: while \cite{Slepcev13} and \cite{Obrie25} prove the soft-penalty topological characterization for arbitrary compactly supported probability measures in the cases $\phi(t) = t^p$ for $p = 1$ and $p \geq 2$ respectively, the proof of the hard-constraint topological characterization \cref{theorem:stepanovtopologicalcharacterization} must crucially use the fact that $\mu$ is not supported on a $1$-dimensional set to avoid examples such as $\Sigma^+$. 
\end{remark}

Now, we generalize the results of \cite{Obrie25}*{Section 1.1.2} on soft-penalty average distance minimizers. The proofs we give are the same as in \cite{Obrie25}: the improvement of our results comes from \cref{theorem:barycentregradient} and \cref{theorem:boundingmassofnoncutpoints}. We begin by proving nontriviality of the barycentre field for soft-penalty minimizers, something which is significantly easier in the soft-penalty case than in the hard-constraint case. Note that \cref{prop:barycentrenontrivialityforsoftpenalty} and \cite{Obrie25}*{Proposition 1.4} are predated by Buttazzo, Manini, and Stepanov's stronger result \cite{Buttazzo09}*{Proposition 2.1} in the case $\phi(t) = t$: the crucial step in our approach for yielding existence of an atom comes from using \cref{cor:boundingmassofnoncutpoints} to show that nontriviality of the barycentre field implies existence of an atom. 
\begin{proposition}[Barycentre nontriviality for soft-penalty minimizers]\label{prop:barycentrenontrivialityforsoftpenalty}
    Assume $\phi \in C^1([0, \infty))$ is nondecreasing. Suppose that $\mu(\Sigma) = 0$ for all $\Sigma \in \mathcal{S}$. Let $\lambda > 0$, and let $\Sigma = \Sigma_{\lambda}$ be a solution to \eqref{eq:softpenaltyadp} with $\mathcal{H}^1(\Sigma) > 0$. Then, for any $\pi_{\Sigma} \in \Pi_{\Sigma}$, $\mathcal{B}_{\pi_{\Sigma}}$ is nontrivial. 
\end{proposition}
\begin{proof}
    Suppose not, and for each $\epsilon > 0$ consider the rescaled set $(1 - \epsilon)\Sigma$. Clearly, $(1 - \epsilon)\Sigma \in \mathcal{S}$, and $\mathcal{H}^1((1 - \epsilon)\Sigma) = \mathcal{H}^1(\Sigma) - \epsilon \mathcal{H}^1(\Sigma)$. Moreover, taking $\xi: \sigma \mapsto - \sigma$, we have that $(1 - \epsilon)\Sigma = \Sigma_{\epsilon, \xi}$ (recalling notation from \cref{theorem:barycentregradient}), and so by \cref{theorem:barycentregradient} and the assumption that $\mathcal{B}_{\pi_{\Sigma}}$ is trivial, 
    \[
    \mathscr{J}(\Sigma) \geq \mathscr{J}((1 -\epsilon)\Sigma) + o(\epsilon).
    \]
    Therefore, 
    \begin{align*}
        (\mathscr{J})^{\lambda}(\Sigma) &= \mathscr{J}(\Sigma) + \lambda \mathcal{H}^1(\Sigma) \\ &\geq \mathscr{J}((1 - \epsilon)\Sigma) + \lambda\mathcal{H}^1((1- \epsilon)\Sigma) + \lambda\epsilon \mathcal{H}^1(\Sigma) + o(\epsilon)\\
        &= (\mathscr{J})^{\lambda}((1 -\epsilon)\Sigma) + \lambda \epsilon \mathcal{H}^1(\Sigma) + o(\epsilon),
    \end{align*}
    and so since $\mathcal{H}^1(\Sigma) >0$, taking $\epsilon$ small enough we conclude that
    \[
    (\mathscr{J})^{\lambda}(\Sigma) > (\mathscr{J})^{\lambda}((1 -\epsilon)\Sigma),
    \]
    contradicting the optimality of $\Sigma$.
\end{proof} 
Using this, we can now prove existence of an atom, and thus the complete topological description, for soft-penalty minimizers. The following result was proven for the case $\phi(t) = t$ in \cite{Slepcev13}*{Lemma 3.1 and 3.3}, and for the case $\phi(t) = t^p$ for $p \geq 2$ in \cite{Obrie25}*{Proposition 1.4 and Theorem 1.2}. Notice that these results are stronger in their respective cases, as they do not require the assumption that $\mu(\Sigma) = 0$ for all $\Sigma \in \mathcal{S}$.
\begin{corollary}[Existence of an atom for soft-penalty minimizers]\label{cor:existenceofanatomforsoftpenaltyminimizers}
    Assume $\phi \in C^1([0, \infty))$ is nondecreasing. Suppose that $\mu(\Sigma) = 0$ for all $\Sigma \in \mathcal{S}$. Let $\lambda > 0$, and let $\Sigma = \Sigma_{\lambda}$ be a solution to \eqref{eq:softpenaltyadp}. Then, $\Sigma$ has an atom. 
\end{corollary}
\begin{proof}
    If $\mathcal{H}^1(\Sigma) = 0$ then the conclusion is trivial, so assume this is not the case. Then, $\Sigma$ has nontrivial barycentre field by \cref{prop:barycentrenontrivialityforsoftpenalty}, and so by \cref{cor:boundingmassofnoncutpoints}, $\Sigma$ has an atom. 
\end{proof}

\cref{theorem:boundingmassofnoncutpoints} allows us to get a slight improvement of Paolini and Stepanov's \cite{Stepanov04}*{Theorem 5.6} in the soft-penalty case, as we can prove absence of loops under only the assumptions that $\phi \in C^1([0, \infty))$ and $\phi$ is nondecreasing, thus removing the assumption of Paolini and Stepanov's property \eqref{eq:alpha2}.

\begin{corollary}[Absence of loops for soft-penalty minimizers]\label{cor:softpenaltyabsenceofloops}
    Assume $\phi \in C^1([0, \infty))$ is nondecreasing. Suppose that $\mu(\Sigma) = 0$ for all $\Sigma \in \mathcal{S}$. Let $\lambda > 0$, and let $\Sigma = \Sigma_{\lambda}$ be a solution to \eqref{eq:softpenaltyadp}. Then, $\Sigma$ contains only finitely many noncut points, and thus contains no loops, i.e. homeomorphic images of $\mathbb{S}^1$.
\end{corollary}
\begin{proof}
The claim is trivial if $\mathcal{H}^1(\Sigma) = 0$, so assume this is not the case. 
    Then, by \cref{prop:barycentrenontrivialityforsoftpenalty} the barycentre field of $\Sigma$ is nontrivial, and so by \cref{cor:boundingmassofnoncutpoints} we get a uniform lower bound on the $\nu$-mass of any noncut point $\sigma \in \Sigma$. But by \cite{Stepanov04}*{Lemma 5.2}, $\mathcal{H}^1$-a.e. point of a loop is a noncut point, and thus since $\nu(\Sigma) = 1$, $\Sigma$ cannot contain any loops. 
\end{proof}

Now, we can combine \cref{cor:existenceofanatomforsoftpenaltyminimizers} with Stepanov's conditional result \cite{Stepanov06}*{Theorem 5.5} to conclude the remainder of the complete topological description for soft-penalty minimizers. 
We define the \textit{order of a point} $\sigma \in \Sigma$ by 
\begin{equation}\label{eq:orddefinition}
    \mathrm{ord}_{\sigma}(\Sigma) := \limsup_{r \to 0^+}\#(\Sigma \cap \partial B_r(\sigma)),
\end{equation}
where $\#$ denotes set cardinality. We say a point $\sigma \in \Sigma$ is a \textit{branching point} if $\mathrm{ord}_{\sigma}(\Sigma) > 2$.

For reference, we state here Paolini and Stepanov's \cite{Stepanov04} property \eqref{eq:alpha2}:
 \begin{quote}
     For every $c > 0$, there is $\lambda = \lambda(c) > 0$ such that 
        \begin{equation}\label{eq:alpha2}\tag{$\alpha_2$}
            |\phi(s) - \phi(t)| \geq \lambda |s - t|
        \end{equation}
        for any $s, t \in [c, \diam \supp \mu]$.
 \end{quote}

\begin{corollary}[Triple branching for soft-penalty minimizers]\label{cor:completetopdescriptionforsoftpenalty}
    Assume $\phi \in C^1([0, \infty))$ is nondecreasing and satisfies Paolini and Stepanov's condition \eqref{eq:alpha2}.
   Suppose that $\mu(\Sigma) = 0$ for all $\Sigma \in \mathcal{S}$. Let $\lambda > 0$, and let $\Sigma = \Sigma_{\lambda}$ be a solution to \eqref{eq:softpenaltyadp}. 
    Then, $\Sigma$ has finitely many branching points, and each branching point $\sigma \in \Sigma$ is a triple branching, i.e. $\mathrm{ord}_{\sigma}(\Sigma) = 3$. 
\end{corollary}

\begin{proof}
    If $\mathcal{H}^1(\Sigma) = 0$, the conclusion is trivial, so assume this is not the case. By \cref{cor:existenceofanatomforsoftpenaltyminimizers}, $\Sigma$ has an atom, and moreover $\Sigma$ solves the hard-constraint average distance problem with $\ell = \mathcal{H}^1(\Sigma)$. So, by \cite{Stepanov06}*{Theorem 5.5 (iii)-(iv)}, $\Sigma$ has finitely many branching points, and each branching point is a triple branching. 
\end{proof}

In particular, the complete topological description (nonexistence of loops, finitely many endpoints, and finitely many branching points with only triple branchings) holds for soft-penalty average distance problem minimizers when we take $\phi(t)= t^p$ for any $p \geq 1$, in any dimension $d \geq 2$. 

Now, we return to our discussion of the hard-constraint average distance problem. 

\section{Atomic noncut points and barycentre nontriviality}\label{sec:atomsandbarycentrefield}

The main result of this section, \cref{theorem:existenceofatom}, is the existence of an atom for (hard-constraint) average distance minimizers. First, in \cref{sec:rescalingofoptimizers}, we will prove that existence of an atom is equivalent to nontriviality of the barycentre field (\cref{cor:nontrivialbarycentrefieldandatoms}), proving along the way an important technical estimate on a typical local modification in the average distance problem (\cref{lemma:estimatingimprovementfromddimcross}). Then, we will prove \cref{theorem:existenceofatom} in \cref{sec:existenceofatoms}. We discuss the consequences of \cref{theorem:existenceofatom} for the structure of average distance minimizers in \cref{sec:topologicaldescription}, concluding with quantitative bounds on the number of noncut points of average distance minimizers in \cref{sec:branchingrates}.

\subsection{Rescaling of optimizers}\label{sec:rescalingofoptimizers}

The goal of this subsection is to prove that the existence of an atom implies nontriviality of the barycentre field, which when combined with \cref{cor:boundingmassofnoncutpoints} will prove the equivalence between existence of an atom and nontriviality of the barycentre field in \cref{cor:nontrivialbarycentrefieldandatoms}. We will do this by studying the \textit{scaling constant}, a constant defined via the barycentre field which measures the change in objective value under the rescaling $\Sigma \mapsto (1 + \epsilon)\Sigma$. 

The main result of this section, \cref{prop:boundingscalingconstantintermsofatoms}, bounds the scaling constant in terms of the $\rho$-mass of points $\sigma^* \in \Sigma$. \cref{prop:boundingscalingconstantintermsofatoms} is essentially a generalization of the “$\nu$ has an atom” case of the proof of nontriviality of the barycentre field by the author, Kobayashi, and Kim in \cite{Obrie25}*{Theorem 3.5}; see \cite{Obrie25}*{Section 4.3} for a discussion of this case. Along the way, we will prove a key technical lemma, \cref{lemma:estimatingimprovementfromddimcross}, which will be used in the proof of existence of an atom (\cref{theorem:existenceofatom}). This lemma bounds the improvement one can gain by adding the $d$\textit{-dimensional cross} to the set $\Sigma$. The $d$-dimensional cross is a typical local modification used when studying the average distance problem, for example by Paolini and Stepanov in \cite{Stepanov04}*{Lemma 3.6} to unconditionally obtain improvements of order $\epsilon^2$ to the objective value, given $\epsilon$ additional budget. Our \cref{lemma:estimatingimprovementfromddimcross} is a version of \cite{Stepanov04}*{Lemma 3.6} specialized for the purpose of proving existence of an atom; our bound is altered to include additional parameters and terms depending on the barycentre field in order to provide the flexibility required to prove \cref{theorem:existenceofatom}. Similar bounds to \cref{lemma:estimatingimprovementfromddimcross} appear in the proof of \cite{Obrie25}*{Theorem 3.5}, and the idea to bound the improvement from the $d$-dimensional cross in terms of the barycentre field originates from Delattre and Fischer's proof of default of self-consistency for the $\phi(t) = t^2$ case of the length-constrained principal curves problem \cite{Delattre17}*{Lemma 3.2}.

We begin by defining the aforementioned scaling constant, and proving its independence on the choice of base point and interpretation in terms of the change in objective value under the rescaling $\Sigma \mapsto (1 + \epsilon)\Sigma$.

\begin{definition}[Scaling constant]\label{Scaling constant}
    Let $\Sigma \subseteq \mathbb{R}^d$ be compact, and let $\pi_{\Sigma} \in \Pi_{\Sigma}$. For any $a \in \mathbb{R}^d$, define the \textit{scaling constant} of $\pi_{\Sigma}$ to be
    \[
    \beta_{\pi_{\Sigma}}^{a} := \int_{\Sigma}(\sigma - a)\cdot \mathcal{B}_{\pi_{\Sigma}}(\sigma)d\nu_{\pi_{\Sigma}}(\sigma).
    \]
\end{definition}

\begin{lemma}[Independence and interpretation of scaling constant]\label{lemma:independenceandinterpretationofscalingconstant}
    Assume $\phi \in C^1([0, \infty))$ is nondecreasing. Let $\Sigma \in \mathcal{S}_{\ell}$ be optimal, and let $\pi_{\Sigma} \in \Pi_{\Sigma}$. Then, $\beta_{\Sigma}:= \beta_{\pi_{\Sigma}}^a$ is defined independently of choice of $a \in \mathbb{R}^d$ and $\pi_{\Sigma} \in \Pi_{\Sigma}$. Moreover, letting $\Sigma_{\epsilon}:= (1+ \epsilon)\Sigma \in \mathcal{S}_{(1 + \epsilon)\ell}$, we have
    \[
    \lim_{\epsilon \to 0^+}\frac{\mathscr{J}(\Sigma) - \mathscr{J}(\Sigma_{\epsilon})}{\epsilon} \geq \beta_{\Sigma}.
    \]
\end{lemma}

\begin{proof}
    First, notice that for any $a \in \mathbb{R}^d$,
    \begin{align*}
        \beta^a_{\pi_{\Sigma}} &= \int_{\Sigma}(\sigma - a)\cdot\mathcal{B}_{\pi_{\Sigma}}(\sigma)d\nu_{\pi_{\Sigma}}(\sigma)\\
        &= \beta^0_{\pi_{\Sigma}} - a\cdot\mathcal{B}^{\mathrm{net}}_{\pi_{\Sigma}} \\
        &= \beta^0_{\pi_{\Sigma}},
    \end{align*}
    since $\mathcal{B}_{\pi_{\Sigma}}^{\mathrm{net}} = 0$ by \cref{lemma:netbarycentrefieldiszero}. By \cref{prop:negligibilityofambiguouslocus}, $\beta_{\pi_{\Sigma}}^0$ is independent of $\pi_{\Sigma} \in \Pi_{\Sigma}$ as well. So, $\beta_{\Sigma}:= \beta_{\pi_{\Sigma}}^a$ is defined independently of choice of $a \in \mathbb{R}^d$ and $\pi_{\Sigma} \in \Pi_{\Sigma}$. Moreover, we have
    \begin{align*}
        \lim_{\epsilon \to 0^+}\frac{\mathscr{J}(\Sigma) - \mathscr{J}(\Sigma_{\epsilon})}{\epsilon} &\geq \lim_{\epsilon \to 0^+}\frac{\mathscr{J}(\pi_{\Sigma}) - \mathscr{J}((1+ \epsilon)\pi_{\Sigma})}{\epsilon} \\ &\geq \int_{\Sigma}\sigma \cdot \mathcal{B}_{\pi_{\Sigma}}(\sigma)d\nu_{\pi_{\Sigma}}(\sigma) \\ &= \beta_{\Sigma},
    \end{align*}
as claimed.
\end{proof}

Now, we define the $d$-dimensional cross.

\begin{definition}[$d$-dimensional cross]\label{def:d-dimensionalcross}
    Given $\epsilon> 0$, define the \textit{$d$-dimensional cross} to be
\[
K_{\epsilon} = \bigcup_{i=1}^d\{te_i \ | \ t \in [-\frac{1}{2d}\epsilon, \frac{1}{2d}\epsilon]\},
\]
where $\{e_1, \dots, e_d\}$ is the standard basis for $\mathbb{R}^d$.
\end{definition}

In our next lemma, we provide a bound on the improvement to the objective value $\mathscr{J}(\Sigma)$ afforded by adding the $d$-dimensional cross to a set $\Sigma$, specialized for the purposes of proving \cref{theorem:existenceofatom}. Recall the notation $\mathscr{J}|_{E}$ from \eqref{eq:restrictedadfnotation}, $\gamma = \gamma_{\pi_{\Sigma}}$ from \eqref{eq:gammadefinition}, and $\rho = \rho_{\pi_{\Sigma}}$ from \eqref{eq:rhodefinition}. Moreover, for $\phi \in C^{1,1}_{\mathrm{loc}}([0, \infty))$, we denote by
\begin{equation}\label{eq:lphinotation}
    L = L_{\phi} := \mathrm{Lip}(\phi'|_{[0, \diam \supp \mu]})
\end{equation}
a Lipschitz constant for $\phi'$ on $[0, \diam \supp \mu]$.

\begin{lemma}[Estimating the improvement from the $d$-dimensional cross]\label{lemma:estimatingimprovementfromddimcross}
        Assume $\phi \in C^{1,1}_{\mathrm{loc}}([0, \infty))$ is strictly increasing. Let $\Sigma \subseteq \mathbb{R}^d$, and let $0 \in \Sigma$. Let $A \subseteq \Sigma$ be Borel with $0 \in A$, and let $\eta = \diam(A)$. Then, for every $\epsilon >0$, we have
        \begin{equation}\label{eq:estimatingimprovementfromddimensionalcross}
        \begin{split} &\mathscr{J}|_{\pi_{\Sigma}^{-1}(A)}(\Sigma) - \mathscr{J}|_{\pi_{\Sigma}^{-1}(A)}(K_{\epsilon}) \\
        & \qquad \geq \bigg\{(\frac{\epsilon}{4d^{3/2}}  - \frac{\eta^2}{32d^{3/2}\epsilon})\rho(A)  -\eta \int_{A}|\mathcal{B}_{\pi_{\Sigma}}(\sigma)|d\nu(\sigma)\\ & \qquad \qquad - (2\eta + \frac{1 + 2 \sqrt{d}}{4d^{3/2}}\epsilon) \gamma(B_{16d^{3/2}\epsilon}(A) \cap \pi_{\Sigma}^{-1}(A)) \\
        & \qquad \qquad - L (\eta + \epsilon)^2\nu(A)\bigg\}. 
        \end{split}  
    \end{equation}
    \end{lemma}
    \begin{proof}
        Fix $A \subseteq \Sigma$ Borel, and fix $\epsilon > 0$. Let $\pi_{\Sigma} \in \Pi_{\Sigma}$ and $\pi_{K_{\epsilon}} \in \Pi_{K_{\epsilon}}$. Fix some $x \in \pi_{\Sigma}^{-1}(A)$. By the mean value theorem, there is some 
        \[\min\{|x - \pi_{K_{\epsilon}}(x)|, | x- \pi_{\Sigma}(x)|\} \leq c_{\epsilon}^x \leq \max\{|x - \pi_{K_{\epsilon}}(x)|, | x- \pi_{\Sigma}(x)|\}\]
        such that
        \[
        \phi(|x - \pi_{\Sigma}(x)|) - \phi(|x - \pi_{K_{\epsilon}}|) \geq (|x - \pi_{\Sigma}(x)| - |x - \pi_{K_{\epsilon}}(x)|)\phi'(c_{\epsilon}^x).
        \]
        Recalling the notation $L = L_{\phi}$ from \eqref{eq:lphinotation}, we have
        \[
        |\phi'(c_{\epsilon}^x)- \phi'(|x - \pi_{\Sigma}(x)|)| \leq L (|\pi_{\Sigma}(x)| + |\pi_{K_{\epsilon}}(x)| )\leq L(\eta + \epsilon),
        \]
        and thus 
        \begin{align*}
            &\phi(|x - \pi_{\Sigma}(x)|) - \phi(| x- \pi_{K_{\epsilon}}(x)| ) \\ & \qquad \geq (|x - \pi_{\Sigma}(x)| - |x - \pi_{K_{\epsilon}}(x)|) \phi'(|x - \pi_{\Sigma}(x)| ) - L(\eta + \epsilon)^2.
        \end{align*}
        Therefore, we see that
        \begin{align*}
            &\mathscr{J}|_{\pi_{\Sigma}^{-1}(A)}(\pi_{\Sigma}) - \mathscr{J}|_{\pi_{\Sigma}^{-1}(A)}(\pi_{K_{\epsilon}}) \\ & \qquad \geq \int_{\pi_{\Sigma}^{-1}(A)}(| x - \pi_{\Sigma}(x)| - |x -\pi_{K_{\epsilon}}(x)|)d\gamma(x) - L (\eta + \epsilon)^2\nu(A).
        \end{align*}
        
        Now, we bound the integral of the difference 
        \begin{equation}\label{eq:differencewewanttoboundddimcross}
            |x - \pi_{\Sigma}(x)| - |x - \pi_{K_{\epsilon}}(x)|.
        \end{equation}
        We will do so by splitting the domain of integration $\pi_{\Sigma}(A)$ into two parts, namely
        \[
        \pi_{\Sigma}^{-1}(A) = (\pi_{\Sigma}^{-1}(A)\setminus B_{16d^{3/2}\epsilon}(A)) \cup (\pi_{\Sigma}^{-1}(A)\cap B_{16d^{3/2}\epsilon}(A)),
        \]
        and bounding this difference separately on each region.

        First, let us bound \eqref{eq:differencewewanttoboundddimcross} in the region $\pi_{\Sigma}^{-1}(A)\setminus B_{16d^{3/2}\epsilon}(A)$. By \cite{Stepanov04}*{Lemma 3.3 (iii)}, we know that for all $x \in \mathbb{R}^d$ with $|x| \geq \frac{\epsilon}{2\sqrt{d}}$, we have
        \begin{equation}\label{eq:xminuspikepislonbound}
             |x - \pi_{K_{\epsilon}(x)}| = \dist(x, K_{\epsilon}) \leq |x| - \frac{\epsilon}{4d^{3/2}}.
        \end{equation}
        In particular, since $16d^{3/2}\epsilon \geq \frac{\epsilon}{2\sqrt{d}}$, the above holds for all $x \in \pi_{\Sigma}^{-1}(A)\setminus B_{16d^{3/2}\epsilon}(A)$. Now, what is left to do is bound the difference between $|x - \pi_{\Sigma}(x)|$ and $|x|$. It is easy to do so using the triangle inequality; however, in order to specialize our bound for the proof of \cref{theorem:existenceofatom}, we will choose a slightly more technical approach which introduces the barycentre field into our bound. Applying \eqref{eq:squarerootinequalityresult} with $G(x) = 0$ and $F(x) = \pi_{\Sigma}(x)$, we have
        \[
        |x| \leq |x - \pi_{\Sigma}(x)| + \pi_{\Sigma}(x)\cdot \frac{(x - \pi_{\Sigma}(x))}{| x- \pi_{\Sigma}(x)|} + \frac{1}{2}\frac{|\pi_{\Sigma}(x)|^2}{|x - \pi_{\Sigma}(x)|}.
        \]
        So, since we are assuming $x \notin B_{16d^{3/2}\epsilon}(A)$, we have \[
        |x - \pi_{\Sigma}(x)| \geq 16d^{3/2}\epsilon.\] 
        Using that $x \in \pi_{\Sigma}^{-1}(A)$ to see that $|\pi_{\Sigma}(x)| \leq \eta$, our bound becomes
        \[
        |x| \leq |x - \pi_{\Sigma}(x)| + \pi_{\Sigma}(x)\cdot \frac{(x - \pi_{\Sigma}(x))}{|x - \pi_{\Sigma}(x)|} + \frac{\eta^2}{32d^{3/2}\epsilon}.
        \]
    Combining this with \eqref{eq:xminuspikepislonbound}, we have
        \[
        |x - \pi_{K_{\epsilon}}(x)| \leq |x - \pi_{\Sigma}(x)| - \frac{\epsilon}{4d^{3/2}} + \pi_{\Sigma}(x)\cdot \frac{(x - \pi_{\Sigma}(x))}{|x - \pi_{\Sigma}(x)|} + \frac{\eta^2}{32d^{3/2}\epsilon}.
        \]
        Thus, we have
        \begin{align*} &\int_{\pi_{\Sigma}^{-1}(A)\setminus B_{16d^{3/2}\epsilon}(A)}(|x -\pi_{\Sigma}(x)| - |x - \pi_{K_{\epsilon}}(x)|)d\gamma(x) \\
         & \qquad \geq  (\frac{\epsilon}{4d^{3/2}}- \frac{ \eta^2}{32 d^{3/2}\epsilon})\gamma(\pi_{\Sigma}^{-1}(A)\setminus B_{16d^{3/2}\epsilon}(A)) \\ &\qquad \qquad - \int_{\pi_{\Sigma}^{-1}(A) \setminus B_{16d^{3/2}\epsilon}(A)}\pi_{\Sigma}(x)\cdot\Delta_{\pi_{\Sigma}}(x)d\mu(x),
        \end{align*}
        where we recall the notation $\Delta_{\pi_{\Sigma}}$ from \eqref{eq:barycentrekernel}.
        Let us simplify this bound slightly. First, we use that
        \[
        \gamma(\pi_{\Sigma}^{-1}(A)\setminus B_{16d^{3/2}\epsilon}(A)) = \rho(A) - \gamma(B_{4d^{3/2}}(A)\cap \pi_{\Sigma}^{-1}(A))
        \]
        to find
        \begin{align*}
            &(\frac{\epsilon}{4d^{3/2}}- \frac{ \eta^2}{32 d^{3/2}\epsilon})\gamma(\pi_{\Sigma}^{-1}(A)\setminus B_{16d^{3/2}\epsilon}(A)) \\
            & \qquad \geq (\frac{\epsilon}{4d^{3/2}}- \frac{ \eta^2}{32 d^{3/2}\epsilon})\rho(A)- \frac{\epsilon}{4d^{3/2}}\gamma(\pi_{\Sigma}^{-1}(A)\setminus B_{16d^{3/2}\epsilon}(A)).
        \end{align*}
        Next, we notice that
        \begin{align*}
        &\int_{\pi_{\Sigma}^{-1}(A) \setminus B_{16d^{3/2}\epsilon}(A)}\pi_{\Sigma}(x)\cdot\Delta_{\pi_{\Sigma}}(x)d\mu(x) \\
        & \qquad = \int_{A}\sigma \cdot \mathcal{B}_{\pi_{\Sigma}}(\sigma)d\nu(\sigma) - \int_{\pi_{\Sigma}^{-1}(A) \cap B_{16d^{3/2}\epsilon}(A)}\pi_{\Sigma}(x)\cdot\Delta_{\pi_{\Sigma}}(x)d\mu(x) \\
        & \qquad \leq \eta \int_{A}|B_{\pi_{\Sigma}}(\sigma)|d\nu(\sigma) + M\eta\gamma(\pi_{\Sigma}^{-1}(A)\cap B_{16d^{3/2}\epsilon}(A));
        \end{align*}
        recall the definition of $M$ from \eqref{eq:mphinotation}. Putting this all together, our bound simplifies to 
        \begin{align*} &\int_{\pi_{\Sigma}^{-1}(A)\setminus B_{16d^{3/2}\epsilon}(A)}(|x -\pi_{\Sigma}(x)| - |x - \pi_{K_{\epsilon}}(x)|)d\gamma(x) \\
        &\qquad \geq  \bigg\{(\frac{\epsilon}{4d^{3/2}}  - \frac{\eta^2}{32d^{3/2}\epsilon})\rho(A) - (\frac{\epsilon}{4d^{3/2}} + \eta)\gamma(B_{16d^{3/2}\epsilon}(A) \cap \pi_{\Sigma}^{-1}(A))  \\
        &\qquad \qquad  -\eta \int_{A}|\mathcal{B}_{\pi_{\Sigma}}(\sigma)|d\nu(\sigma)\bigg\}.
        \end{align*}

        In the region $\pi_{\Sigma}^{-1}(A)\cap B_{16d^{3/2}\epsilon}(A)$, our bound on \eqref{eq:differencewewanttoboundddimcross} can be more straightforward. For any $x \in \mathbb{R}^d$, we have
        \begin{align*}
            |x - \pi_{\Sigma}(x)| - |x - \pi_{K_{\epsilon}}(x)| &\geq - |\pi_{\Sigma}(x) - \pi_{K_{\epsilon}}(x)| \\ &\geq -(|\pi_{\Sigma}(x)| + |\pi_{K_{\epsilon}}(x)|) \\ &\geq -(\eta + \frac{\epsilon}{2d}),
        \end{align*}
        so we see that
        \begin{align*} &\int_{\pi_{\Sigma}^{-1}(A)\cap B_{16d^{3/2}\epsilon}(A)}(| x- \pi_{\Sigma}(x)| - |x - \pi_{K_{\epsilon}}(x)|)d\gamma(x)\\
        &\qquad  \geq -(\eta + \frac{\epsilon}{2d})\gamma(B_{16d^{3/2}\epsilon}(A)\cap \pi_{\Sigma}^{-1}(A)).
        \end{align*}
    Thus, since 
    \[
    \frac{\epsilon}{4d^{3/2}} + \frac{\epsilon}{2d} = \frac{1 + 2\sqrt{d}}{4d^{3/2}}\epsilon,
    \]
 combining our bounds yields
    \begin{align*} &\mathscr{J}|_{\pi_{\Sigma}^{-1}(A)}(\Sigma) - \mathscr{J}|_{\pi_{\Sigma}^{-1}(A)}(K_{\epsilon}) \\
        & \qquad \geq \bigg\{(\frac{\epsilon}{4d^{3/2}}  - \frac{\eta^2}{32d^{3/2}\epsilon})\rho(A)  -\eta \int_{A}|\mathcal{B}_{\pi_{\Sigma}}(\sigma)|d\nu(\sigma)\\ & \qquad \qquad - (2\eta + \frac{1 + 2 \sqrt{d}}{4d^{3/2}}\epsilon) \gamma(B_{16d^{3/2}\epsilon}(A) \cap \pi_{\Sigma}^{-1}(A)) \\
        & \qquad \qquad - L (\eta + \epsilon)^2\nu(A)\bigg\},
    \end{align*}
    as was claimed.
    \end{proof}

    \begin{remark}
        The reason that we use $B_{16d^{3/2}\epsilon}(A)$ to divide the region $\pi_{\Sigma}^{-1}(A)$ is the following: if we take $A = B_{\epsilon}(0)$, then we have that $\eta = 2 \epsilon$. Thus, we see that
        \[
        \frac{\epsilon}{4d^{3/2}} - \frac{\eta^2}{32d^{3/2}\epsilon} \geq \frac{\epsilon}{8d^{3/2}} > 0.
        \]
        The fact that this difference is a positive constant times $\epsilon$ is a result of the choice to use $B_{16d^{3/2}\epsilon}(A)$, and plays a key role in the proof of \cref{theorem:existenceofatom}.
    \end{remark}

    Now, we proceed with the main theorem of this subsection, which provides a lower bound on the scaling constant in terms of the $\rho$-mass of any point $\sigma^* \in \Sigma$, and in particular in terms of any atom of $\Sigma$, should one exist. 

    \begin{theorem}[Bounding the scaling constant in terms of atoms]\label{prop:boundingscalingconstantintermsofatoms} Assume $\phi \in C^{1,1}_{\mathrm{loc}}([0, \infty))$ is strictly increasing. Suppose $\Sigma \in \mathcal{S}_{\ell}$ is optimal for some $\ell > 0$. Then, for every point $\sigma^* \in \Sigma$, we have  \begin{equation}\label{eq:boundingbarycentrefieldintermsofatom}
         \beta_{\Sigma} \geq \frac{\ell}{4d^{3/2}}\rho_{\pi_{\Sigma}}\{\sigma^*\}.
    \end{equation}
    \end{theorem}
    \begin{proof}
        Take some $\sigma^* \in \Sigma$, and using the translation invariance of the average distance problem assume without loss of generality that $\sigma^* = 0$. Let $\epsilon > 0$. Defining 
        \[
        \Sigma_{\epsilon} = (1 - \frac{\epsilon}{\ell})\Sigma,
        \]
        we have
        \begin{align*}
            \lim_{\epsilon \to 0^+}\frac{\mathscr{J}(\Sigma) - \mathscr{J}(\Sigma_{\epsilon})}{\epsilon} &\geq \lim_{\epsilon \to 0^+}\frac{\mathscr{J}(\pi_{\Sigma}) - \mathscr{J}((1 - \frac{\epsilon}{\ell})\pi_{\Sigma})}{\epsilon} \\
            &\geq - \frac{1}{\ell}\int_{\Sigma}\sigma \cdot \mathcal{B}_{\pi_{\Sigma}}d\nu_{\pi_{\Sigma}} \\
            &= - \frac{1}{\ell}\beta_{\Sigma}.
        \end{align*}
        Now, define 
        \[
        \Sigma'_{\epsilon} = \Sigma_{\epsilon}\cup K_{\epsilon},
        \]
        then $\Sigma'_{\epsilon}$ is compact, connected, and 
        \[
        \mathcal{H}^1(\Sigma'_{\epsilon}) \leq \mathcal{H}^1(\Sigma_{\epsilon}) + \mathcal{H}^1(K_{\epsilon})\leq \ell.
        \]
        So, by the optimality of $\Sigma$, we have
        \[
        0 \geq \mathscr{J}(\Sigma) - \mathscr{J}(\Sigma_{\epsilon}').
        \]
        Now, take any $\pi_{\Sigma_{\epsilon}} \in \Pi_{\Sigma_{\epsilon}}$, and define
        \[
        F_{\epsilon}(x):= \bigg \{ \begin{matrix}
            \pi_{\Sigma_{\epsilon}}(x), & x \notin \pi_{\Sigma}^{-1}\{0\}, \\
            \pi_{K_{\epsilon}}(x), & x \in \pi_{\Sigma}^{-1}\{0\}.
        \end{matrix}
        \]
        Then, $\img(F_{\epsilon}) \subseteq \Sigma_{\epsilon}'$, and as a result,
        \begin{align*}
            \mathscr{J}(\Sigma_{\epsilon}) - \mathscr{J}(\Sigma_{\epsilon}') &\geq \mathscr{J}(\pi_{\Sigma_{\epsilon}}) - \mathscr{J}(F_{\epsilon}) \\
            &= \mathscr{J}|_{\pi_{\Sigma}^{-1}\{0\}}(\pi_{\Sigma_{\epsilon}}) - \mathscr{J}|_{\pi_{\Sigma}^{-1}\{0\}}(\pi_{K_{\epsilon}}).
        \end{align*}
        So, applying \cref{lemma:estimatingimprovementfromddimcross} with $A = \{0\}$, we see since $\eta = 0$ that
        \[
        \mathscr{J}(\Sigma_{\epsilon})- \mathscr{J}(\Sigma_{\epsilon}') \geq \frac{\epsilon}{4d^{3/2}}\rho_{\pi_{\Sigma_{\epsilon}}}(\{0\}) +o(\epsilon),
        \]
        where we are using that $\gamma(B_{16d^{3/2}\epsilon}(0))  = o(1)$ since $\mu$ has no atoms. Putting this all together, we see that
        \begin{align*}
            0 & \geq \mathscr{J}(\Sigma) - \mathscr{J}(\Sigma_{\epsilon}') \\ &= \mathscr{J}(\Sigma) - \mathscr{J}(\Sigma_{\epsilon}) + \mathscr{J}(\Sigma_{\epsilon}) - \mathscr{J}(\Sigma'_{\epsilon}) \\
            &\geq - \frac{\epsilon}{\ell}\beta_{\Sigma} + \frac{\epsilon}{4d^{3/2}}\rho_{\pi_{\Sigma_{\epsilon}}}(\{0\}) + o(\epsilon),
        \end{align*}
        so 
        \begin{equation}\label{eq:betalowerboundintermediate}
            \beta_{\Sigma} \geq \limsup_{\epsilon \to 0^+}\frac{\ell}{4d^{3/2}}\rho_{\pi_{\Sigma_{\epsilon}}}(\{0\}).
        \end{equation}
        
        We now claim that we can choose $\pi_{\Sigma_{\epsilon}} \in \Pi_{\Sigma_{\epsilon}}$ for each $\epsilon > 0$ so that
        \[
        \limsup_{\epsilon \to 0^+}\rho_{\pi_{\Sigma_{\epsilon}}}\{0\} \geq \rho_{\pi_{\Sigma}}\{0\}.
        \]
        Indeed, for each $\epsilon > 0$, define 
        \[
        \Phi_{\epsilon} = \{x\in \mathbb{R}^d \ | \ \dist(x, \Sigma_{\epsilon}) = \dist(x, 0)\},
        \]
        where here $\Sigma_0 = \Sigma$. For any $\epsilon >0$, we may use measurable selection \cref{lemma:measurableselection} to take some $\pi_{\Sigma_{\epsilon}} \in \Pi_{\Sigma_{\epsilon}}$ such that $\pi_{\Sigma_{\epsilon}}^{-1}\{0\} = \Phi_{\epsilon}$. So, for any sequence of positive real numbers $\epsilon_n$ converging to $0$, we have 
        \begin{align*}
            \limsup_{\epsilon \to 0^+} \rho_{\pi_{\Sigma_{\epsilon}}}\{0\} &= \limsup_{n \to \infty} \int_{\mathbb{R}^d}\chi_{\Phi_{\epsilon_n}}\phi'(\dist(x, \Sigma_{\epsilon}))d \mu \\
            & \geq \liminf_{n \to \infty} \int_{\mathbb{R}^d}\chi_{\Phi_{\epsilon_n}}\phi'(\dist(x, \Sigma_{\epsilon}))d \mu \\ 
            & \geq \int_{\mathbb{R}^d}\liminf_{n \to \infty}\chi_{\Phi_{\epsilon_n}}\phi'(\dist(x, \Sigma_{\epsilon}))d \mu,
        \end{align*}
        where the last inequality uses Fatou's lemma; this is justified since $\phi'$ is assumed to be nonnegative. 
        Notice that for any $\epsilon_1, \epsilon_2 \geq 0$, we can use the triangle inequality to get the bound
        \begin{equation}\label{eq:distxsigmacontinuity}
            |\dist(x, \Sigma_{\epsilon_1}) - \dist(x, \Sigma_{\epsilon_2})| \leq |\epsilon_1 - \epsilon_2| \frac{1}{\ell}\dist(x, \Sigma).
        \end{equation}
        Therefore, since $\phi'$ is continuous, we know that the family of functions $\phi'(\dist(x, \Sigma_{\epsilon_n}))$ converge pointwise to $\phi'(\dist(x, \Sigma))$. In particular, 
        \[
        \liminf_{n \to \infty}\chi_{\Phi_{\epsilon_n}}\phi'(\dist(x, \Sigma_{\epsilon})) = \chi_{\liminf_n \Phi_{\epsilon_n}}\phi'(\dist(x, \Sigma)),
        \]
        where 
        \[
        \liminf_n \Phi_{\epsilon_n} := \bigcup_{N \in \mathbb{N}}\bigcap_{n > N}\Phi_{\epsilon_n}.
        \]
        Now, we claim that $\liminf_{n}\Phi_{\epsilon_n} \supseteq \Phi_0$. Indeed, for each $x$, the map $\epsilon \mapsto \dist(x, \Sigma_{\epsilon})$ is continuous by \eqref{eq:distxsigmacontinuity}. In particular, if $\dist(x, \Sigma_{\epsilon})=\dist(x, 0)$ for all sufficiently small $\epsilon > 0$, then $\dist(x, \Sigma) = \dist(x, 0)$. So, we have
        \[
        \liminf_{n}\Phi_{\epsilon_n} \supseteq \Phi_0,
        \]
        and thus since $\phi'(\dist(x, \Sigma)) \geq 0$, 
        \begin{align*}
            \limsup_{\epsilon \to 0^+} \rho_{\pi_{\Sigma_{\epsilon}}}\{0\} &\geq \int_{R^{d}}\chi_{\Phi_0}\phi'(\dist(x, \Sigma))d \mu \\
            &=\gamma_{\pi_{\Sigma}}(\Phi_0).
        \end{align*}
        Finally, from \cref{prop:negligibilityofambiguouslocus} and the fact that $\gamma \ll \mu$, we know that 
        \[
        \gamma(\Phi_0 \setminus \pi_{\Sigma}^{-1}\{0\}) = 0,
        \]
        upon which we conclude that 
        \[
        \limsup_{\epsilon \to 0^+} \rho_{\pi_{\Sigma_{\epsilon}}}\{0\} \geq \rho_{\pi_{\Sigma}}(\{0\}).
        \]
        Combining this with \eqref{eq:betalowerboundintermediate}, we see that
         \[
         \beta_{\Sigma} \geq \frac{\ell}{4d^{3/2}}\rho_{\pi_{\Sigma}}\{0\},
         \]
         as claimed. 
    \end{proof}
    
\begin{remark}\label{remark:scalingandlipschitzinequality}
    Notice that for any noncut point $\sigma^* \in \Sigma$, since the identity map is 1-Lipschitz,
    \begin{align*}
        \beta_{\Sigma} &= \int_{\Sigma} (\sigma - \sigma^*) \cdot \mathcal{B}_{\pi_{\Sigma}}(\sigma) d\nu_{\pi_{\Sigma}}(\sigma) \\
        &\leq \max_{\xi \in \mathrm{Lip}^*(\Sigma)}\int_{\Sigma} \xi(\sigma)\cdot \mathcal{B}_{\pi_{\Sigma}}(\sigma)d\nu_{\pi_{\Sigma}}(\sigma).
    \end{align*}
\end{remark}

Finally, we combine \cref{prop:boundingscalingconstantintermsofatoms} with \cref{cor:boundingmassofnoncutpoints} to prove the equivalence between nontriviality of the barycentre field and existence of an atom for average distance minimizers.

\begin{corollary}[Nontrivial barycentre field and atoms]\label{cor:nontrivialbarycentrefieldandatoms}
    Assume $\phi \in C^{1,1}_{\mathrm{loc}}([0, \infty))$ is strictly increasing. Let $\Sigma \in \mathcal{S}_{\ell}$ be optimal, and let $\pi_{\Sigma} \in \Pi_{\Sigma}$. Then, the following are equivalent:
    \begin{enumerate}[label = (\roman*)]
        \item $\mathcal{B}_{\pi_{\Sigma}}$ is nontrivial,
        \item Every noncut point $\sigma^* \in \Sigma$ is an atom, i.e. $\nu_{\pi_{\Sigma}}\{\sigma^*\} > 0$,
        \item There exists a point $\sigma^* \in \Sigma$ with $\nu_{\pi_{\Sigma}}\{\sigma^*\} > 0$,
        \item There exists $\sigma^* \in \Sigma$ with $\rho_{\pi_{\Sigma}}\{\sigma^*\} > 0$,
        \item $\beta_{\pi_{\Sigma}} > 0$.
    \end{enumerate}
\end{corollary}

\begin{proof}
    ((i) $\implies$ (ii)) If $\mathcal{B}_{\pi_{\Sigma}}$ is nontrivial, then by \cref{cor:boundingmassofnoncutpoints}, there exists some constant $\lambda > 0$ such that for every noncut point $\sigma^* \in \Sigma$, 
    \[
    |\mathcal{B}_{\pi_{\Sigma}}(\sigma^*)|\nu_{\pi_{\Sigma}}\{\sigma^*\} \geq \lambda \int_{\Sigma}|\mathcal{B}_{\pi_{\Sigma}}(\sigma)|d\nu_{\pi_{\Sigma}}(\sigma) > 0.
    \]

    ((ii) $\implies$ (iii)) By  \cite{kuratowski}*{§47 Theorem IV.5}, $\Sigma$ always contains at least two noncut points.

    ((iii) $\implies$ (iv)) Suppose $\sigma^* \in \Sigma$ with $\nu_{\pi_{\Sigma}}\{\sigma^*\} > 0$. Then, $\pi_{\Sigma}^{-1}\{\sigma^*\}$ is a set with positive $\mu$-mass. Since $\phi$ is strictly increasing and $\mu(\Sigma) = 0$ by assumption, we have that $\phi'(|x - \pi_{\Sigma}(x)|) > 0$ $\mu$-a.e., and therefore 
    \[
    \rho_{\pi_{\Sigma}}\{\sigma^*\} = \int_{\pi_{\Sigma}^{-1}\{\sigma^*\}} \phi'(|x - \pi_{\Sigma}(x)|) d \mu(x) > 0.
    \]

    ((iv) $\implies$ (v)) By \cref{prop:boundingscalingconstantintermsofatoms}, we have
    \[
    \beta_{\Sigma} \geq \frac{\ell}{4d^{3/2}}\rho_{\pi_{\Sigma}}\{\sigma^*\} > 0.
    \]
    ((iv) $\implies$ (i)) is obvious.
\end{proof}

\subsection{Existence of atoms}\label{sec:existenceofatoms}

We now proceed with the main theorem of this section: the existence of atoms for average distance minimizers, provided $\phi \in C^{1,1}_{\mathrm{loc}}([0, \infty))$ satisfies \eqref{eq:alpha2} and is strictly increasing. This generalizes the results of \cite{Obrie25}, where existence of an atom was shown for $\phi(t) = t^p$ with $p =2$ or $p > \frac{1}{2}(3 + \sqrt{5})$; in particular, \cref{theorem:existenceofatom} applies to the prototypical $\phi(t) = t$ case of the average distance problem, as well as $\phi(t) = t^p$ for all $p \geq 2$.

The proofs of both \cref{theorem:existenceofatom} and \cite{Obrie25}*{Theorem 3.5} proceed by contradiction. Supposing we are given a minimizer $\Sigma \in \mathcal{S}_{\ell}$ with no atom, we wish to modify $\Sigma$ in such a way that we gain back $\epsilon$ length budget, while controlling the increase to the objective value using the barycentre field. By \cref{cor:nontrivialbarycentrefieldandatoms}, the assumption that $\Sigma$ has no atoms implies that $\Sigma$ has trivial barycentre field, so the increase to the objective value from the budget-gaining modification only depends on the higher order terms in the barycentre field expansion. If we add back the gained $\epsilon$ length in a clever way, we can hope to get a decrease to the objective value which outweighs these higher order terms; this would then allows us to construct a competitor with strictly smaller objective value than that of $\Sigma$, a contradiction. 

The crux of this argument lies in adding back the recovered $\epsilon$ length in such a way that the improvement to the objective value outweighs the increase from the budget-gaining modification. In particular, it is crucial to control the higher order terms in the barycentre field expansion for the budget-gaining modification. This is one of the major obstructions to generalizing the strategy of proof of \cite{Obrie25}*{Theorem 3.5} to, for example, the case $\phi(t) = t$; indeed, \cite{Obrie25} relies on a global rescaling to gain back length, whose higher-order terms are much more difficult to control when the map $t \mapsto \phi(|t|)$ is no longer $C^{1}$ at the origin. Instead, we will rely on a local modification to gain back $\epsilon$ length from $\Sigma$, as the higher-order terms coming from such a modification can be bounded sufficiently well using \cref{prop:boundinghigherorderestimatescase1}. In addition to simplifying the proof, this local argument allows us to avoid the issues caused by lower regularity of $t \mapsto \phi(|t|)$ to the proof strategy of \cite{Obrie25}*{Theorem 3.5}, allowing us to prove existence of atoms in cases that were previously out of reach.

\begin{theorem}[Minimizers have atoms]\label{theorem:existenceofatom}
     Assume $\phi \in C^{1,1}_{\mathrm{loc}}([0, \infty))$ is strictly increasing and satisfies \eqref{eq:alpha2}. Assume that $\mu(B_{\epsilon}(x)) = o(\epsilon)$ for each $x \in \mathbb{R}^d$. Then, any minimizer $\Sigma \in \mathcal{S}_{\ell}$ has an atom. 
\end{theorem}
\begin{proof}
    Let $\Sigma \in \mathcal{S}_{\ell}$, and suppose for the sake of contradiction that $\nu$ has no atoms. Let $\sigma_1, \sigma_2 \in \Sigma$ be two distinct noncut points. Write $\sigma^* = \sigma_1$, and take $\sigma_2 = 0$ to be the origin. By \cref{lemma:noncutneighbourhood}, take $\{B_n\}_{n \in \mathbb{N}}$ to be a noncut neighbourhood system for $\sigma^*$, and let $\epsilon_n = \dist(\sigma^*, \partial B_n)$, where the boundary is taken with respect to the subspace topology on $\Sigma$. Without loss of generality, assume that \begin{equation}\label{eq:assumptiononnoncutpointcomparison}
        \limsup_{n \to \infty}\frac{\rho(B_{\epsilon_n}(0))}{\rho(B_{\epsilon_n}(\sigma^*))} > 0;
    \end{equation}
    if this is not the case then we may interchange the roles of $\sigma^*$ and $0$. Notice that 
    \[
    \frac{\rho(B_{\epsilon_n}(0))}{\rho(B_n)} \geq \frac{\rho(B_{\epsilon_n}(0))}{\rho(B_{\epsilon_n}(\sigma^*))}
    \]
    for each $n$. Using \eqref{eq:assumptiononnoncutpointcomparison} and passing to a subsequence, we may assume that
    \begin{equation}\label{eq:Cdefinition}
         C := \lim_{n \to \infty}\frac{\rho(B_{\epsilon_n}(0))}{\rho(B_n)}
    \end{equation}
    converges, and $C > 0$. Passing to a further subsequence, assume that $B_n \cap B_{\epsilon_n}(0) = \emptyset$ for each $n \in \mathbb{N}$. For each $n$, define
    \[
    \Sigma'_n= (\Sigma \setminus B_n)\cup K_{\epsilon_n}.
    \]
    Then, $\Sigma'_n$ is compact, connected, and $\mathcal{H}^1(\Sigma'_n)\leq \ell$. We begin with a technical lemma which will form the basis for our proof. 
    \begin{lemma}[Conditional bound]\label{lemma:conditionalbound}
        Take the assumptions in \cref{theorem:existenceofatom}. Suppose $\Sigma \in \mathcal{S}_{\ell}$ is a minimizer, and $\Sigma$ has no atom. Then, for any $n \in \mathbb{N}$, for any $\delta > 0$ possibly depending on $n$, we have
        \begin{equation}\label{eq:conditionalboundequation}
            \begin{split}
                \mathscr{J}(\Sigma) - \mathscr{J}(\Sigma'_n) &\geq \frac{\epsilon_n}{8d^{3/2}}\rho(B_{\epsilon_n}(0)) - \frac{\epsilon_n^2}{2\delta}\rho(B_n) \\ &\qquad - 2M\epsilon_n \mu(B_{\delta}(\Sigma)\cap \pi_{\Sigma}^{-1}(\overline{B}_n)) + o(\epsilon_n^2).
            \end{split}
        \end{equation}
    \end{lemma}
    \begin{proof}
        Taking the boundary $\partial B_n$ with respect to the subspace topology on $\Sigma$, define
    \[
    F_n(x) := \bigg \{ \begin{matrix}
        \pi_{K_{\epsilon_n}}(x),
        & x \in \pi_{\Sigma}^{-1}(B_{\epsilon_n}(0)) \\
        \pi_{\partial B_n}(x), & x \in \pi_{\Sigma}^{-1}(\overline{B}_n), \\
        \pi_{\Sigma}(x), & \text{otherwise}.
    \end{matrix}
    \]
    Then, we see that $\img(F_n) \subseteq \Sigma_n'$, and therefore
    \begin{align*}
         \mathscr{J}(\Sigma) - \mathscr{J}(\Sigma'_n) &\geq \mathscr{J}(\pi_{\Sigma}) -\mathscr{J}(F_n) \\
         &\geq \mathscr{J}|_{\pi_{\Sigma}^{-1}(B_{\epsilon_n}(0))}(\Sigma) - \mathscr{J}|_{\pi_{\Sigma}^{-1}(B_{\epsilon_n}(0))}(K_{\epsilon_n})\\
         &\qquad + \mathscr{J}|_{\pi_{\Sigma}^{-1}(\overline{B}_n)}(\Sigma) - \mathscr{J}|_{\pi_{\Sigma}^{-1}(\overline{B}_n)}(\partial B_n).
    \end{align*}
    
    First, we use \cref{lemma:estimatingimprovementfromddimcross} to bound the term \[\mathscr{J}|_{\pi_{\Sigma}^{-1}(B_{\epsilon_n}(0))}(\Sigma) - \mathscr{J}|_{\pi_{\Sigma}^{-1}(B_{\epsilon_n}(0))}(K_{\epsilon_n}).\] Indeed, take $A = B_{\epsilon_n}(0) \cap \Sigma$, and notice that $\eta= \diam(A) \leq 2\epsilon_n$. So, by \cref{lemma:estimatingimprovementfromddimcross}, we have
    \begin{align*} &\mathscr{J}|_{\pi_{\Sigma}^{-1}(B_{\epsilon_n}(0))}(\Sigma) - \mathscr{J}|_{\pi_{\Sigma}^{-1}(B_{\epsilon_n}(0))}(K_{\epsilon_n}) \\
    &\qquad \geq \bigg\{\frac{\epsilon_n}{8d^{3/2}}\rho(B_{\epsilon_n}(0))- 2\epsilon_n \int_{B_{\epsilon_n}(0)\cap \Sigma}|\mathcal{B}_{\pi_{\Sigma}}(\sigma)|d\nu_{\pi_{\Sigma}}(\sigma)\\
    &\qquad \qquad - (\frac{1 + 2 \sqrt{d} + 16d^{3/2}}{4d^{3/2}})\epsilon_n \gamma(B_{16d^{3/2}\epsilon_n}(\Sigma) \cap \pi_{\Sigma}^{-1}(B_{\epsilon_n}(0))) \\
    &\qquad \qquad  - 9L\epsilon_n^2\nu(B_{\epsilon_n}(0))\bigg\}.
    \end{align*}
    The assumption that $\nu$ is atomless implies that $\mathcal{B}_{\pi_{\Sigma}}$ is trivial by \cref{theorem:boundingmassofnoncutpoints}, and therefore
    \[
    2\epsilon_n \int_{B_{\epsilon_n}(0)\cap \Sigma}|\mathcal{B}_{\pi_{\Sigma}}(\sigma)|d\nu_{\pi_{\Sigma}}(\sigma) = 0.
    \]
    Since $\gamma(B_{\epsilon}(0)) = o(\epsilon)$ by assumption, we have
   \[
   (\frac{1 + 2 \sqrt{d} + 16d^{3/2}}{4d^{3/2}})\epsilon_n \gamma(B_{16d^{3/2}\epsilon_n}(\Sigma) \cap \pi_{\Sigma}^{-1}(B_{\epsilon_n}(0))) = o(\epsilon_n^2).
   \]
   Finally, since we are assuming $\nu$ is atomless, $\nu(B_{\frac{\epsilon_n}{8d^{3/2}}}(0)) = o(1)$, so we have
   \[
   9L\epsilon_n^2\nu(B_{\epsilon_n}(0)) = o(\epsilon_n^2).
   \]
   Thus, we may write our bound as  \begin{equation}\label{eq:1atomsexistfirstbound1} \mathscr{J}|_{\pi_{\Sigma}^{-1}(B_{\epsilon_n}(0))}(\Sigma) - \mathscr{J}|_{\pi_{\Sigma}^{-1}(B_{\epsilon_n}(0))}(K_{\epsilon_n}) \geq \frac{\epsilon_n}{8d^{3/2}}\rho(B_{\epsilon_n}(0)) + o(\epsilon_n^2).
    \end{equation}

    Now, we use \cref{prop:boundinghigherorderestimatescase1} to bound the term $\mathscr{J}|_{\pi_{\Sigma}^{-1}(\overline{B}_n)}(\Sigma) - \mathscr{J}|_{\pi_{\Sigma}^{-1}(\overline{B}_n)}(\partial B_n)$. Indeed, noticing that
    \[
    \mathscr{J}|_{\pi_{\Sigma}^{-1}(\overline{B}_n)} =\mathscr{J}^{\mu\llcorner \pi_{\Sigma}^{-1}(\overline{B}_n)},
    \]
    we may apply \cref{prop:boundinghigherorderestimatescase1} to see that for any $\delta > 0$,
   \begin{align*} &\mathscr{J}|_{\pi_{\Sigma}^{-1}(\overline{B}_n)}(\Sigma) - \mathscr{J}|_{\pi_{\Sigma}^{-1}(\overline{B}_n)}(\partial B_n) \\
   &\qquad \geq \mathscr{J}|_{\pi_{\Sigma}^{-1}(\overline{B}_n)}(\pi_{\Sigma}) - \mathscr{J}|_{\pi_{\Sigma}^{-1}(\overline{B}_n)}(\pi_{\partial B_n}\circ \pi_{\Sigma}) \\
   &\qquad \geq \bigg\{\int_{\overline{B_n}}(\pi_{\partial B_n}(\sigma) - \sigma)\cdot \mathcal{B}_{\pi_{\Sigma}}(\sigma)d\nu(\sigma)  - \frac{\epsilon_n^2}{2\delta}\gamma(\pi_{\Sigma}^{-1}(\overline{B}_n)\setminus B_{\delta}(\Sigma)) \\
   & \qquad \qquad  - L\epsilon_n^2 \mu(\pi_{\Sigma}^{-1}(\overline{B}_n)\setminus B_{\delta}(\Sigma))  - 2M\epsilon_n \mu(B_{\delta}(\Sigma)\cap \pi_{\Sigma}^{-1}(\overline{B}_n))\bigg\}.
   \end{align*}
    By \cref{theorem:boundingmassofnoncutpoints}, the assumption that $\nu$ has no atoms implies that the barycentre field is trivial. Moreover, since $\nu(B_n) \to 0$ as $n \to \infty$, we have that $L\epsilon_n^2 \nu(B_n) = o(\epsilon_n)^2$. Therefore, our bound simplifies to 
    \begin{equation}\label{eq:2existenceofatomsbound2}
    \begin{split} &\mathscr{J}|_{\pi_{\Sigma}^{-1}(\overline{B}_n)}(\Sigma) - \mathscr{J}|_{\pi_{\Sigma}^{-1}(\overline{B}_n)}(\partial B_n)
   \\
   &\qquad \geq - \frac{\epsilon_n^2}{2\delta}\rho(B_n) - o(\epsilon_n^2) - 2M\epsilon_n \mu(B_{\delta}(\Sigma)\cap \pi_{\Sigma}^{-1}(\overline{B}_n)).
    \end{split}
   \end{equation}
    Combining \eqref{eq:1atomsexistfirstbound1} and \eqref{eq:2existenceofatomsbound2}, we conclude that
    \begin{align*}
        \mathscr{J}(\Sigma) - \mathscr{J}(\Sigma'_n) &\geq \frac{\epsilon_n}{8d^{3/2}}\rho(B_{\epsilon_n}(0)) - \frac{\epsilon_n^2}{2\delta}\rho(B_n) \\ &\qquad - 2M\epsilon_n \mu(B_{\delta}(\Sigma)\cap \pi_{\Sigma}^{-1}(\overline{B}_n)) + o(\epsilon_n^2).
    \end{align*}
    which is the desired bound.
    \end{proof}
   
    Finally, we may use the bound \eqref{eq:conditionalboundequation} to contradict the minimality of $\Sigma$. To do so, we choose $\delta$ in terms of $\epsilon_n$ so that the right hand side of \eqref{eq:conditionalboundequation} is positive for sufficiently large $n$. Recall the definition of $C$ from \eqref{eq:Cdefinition}. Since $C > 0$, we may take 
    \[
    \delta = \frac{8d^{3/2}}{C}\epsilon_n.
    \]
    Since $\mu(B_{\epsilon}(0)) = o(\epsilon)$, we conclude that
    \[
    2M\epsilon_n \mu(B_{\frac{16d^{3/2}}{C}\epsilon_n}(\Sigma) \cap \pi_{\Sigma}^{-1}(\overline{B}_n)) = o(\epsilon_n^2).
    \]
    Moreover, by the definition of $C$ \eqref{eq:Cdefinition}, we have
    \[
    \frac{C}{2}\rho(B_n) \leq \rho(B_{\epsilon_n}(0))
    \]
    for all $n$ sufficiently large. So, substituting our choice of $\delta$ into \eqref{eq:conditionalboundequation}, we get that for all sufficiently large $n$,
    \[
    \mathscr{J}(\Sigma) - \mathscr{J}(\Sigma_n') \geq \frac{\epsilon_n}{16d^{3/2}}\rho(B_{\epsilon_n}(0)) + o(\epsilon_n^2).
    \]
    So, since $\lim_{n \to \infty}\frac{\rho(B_{\epsilon_n}(0))}{\epsilon_n} > 0$ by \cref{lemma:noncutneighbourhoodshavepositivemass}, if we choose $n$ to be large enough we find that
    \[
    \mathscr{J}(\Sigma) - \mathscr{J}(\Sigma'_n) > 0,
    \]
    contradicting the minimality of $\Sigma$.
\end{proof}

\subsection{Topological description of average distance minimizers}\label{sec:topologicaldescription}

By combining \cref{theorem:existenceofatom} with the conditional results of \cite{Stepanov06}, we may now prove the complete topological description of average distance minimizers when $\phi \in C^{1,1}_{\mathrm{loc}}([0, \infty))$ is strictly increasing and satisfies \eqref{eq:alpha2}. Recall from \eqref{eq:orddefinition} that 
\[
\mathrm{ord}_{\sigma}(\Sigma) := \limsup_{r \to 0^+}\#(\Sigma \cap \partial B_r(\sigma)).
\]

\begin{theorem}[Complete topological description]\label{theorem:stepanovtopologicalcharacterization}
    Assume $\phi \in C^{1,1}_{\mathrm{loc}}([0, \infty))$ is strictly increasing and satisfies \eqref{eq:alpha2}. Assume that $\mu(B_{\epsilon}(x)) = o(\epsilon)$ for each $x \in \mathbb{R}^d$. Then, any minimizer $\Sigma \in \mathcal{S}_{\ell}$ satisfies the following:
    \begin{enumerate}
        \item $\Sigma$ does not contain any homeomorphic images of $\mathbb{S}^1$, and in particular any noncut point of $\Sigma$ is an endpoint \cite{Stepanov04}*{Theorem 5.6}.
        \item The number of noncut points of $\Sigma$ is finite.
        \item There are finitely many branching points, i.e. points $x \in \Sigma$ with $\mathrm{ord}_x\Sigma > 2$, of $\Sigma$.
        \item Every branching point $x$ is a triple junction, i.e. $\mathrm{ord}_x \Sigma = 3$.
    \end{enumerate}
\end{theorem}
\begin{proof}
    By \cref{theorem:existenceofatom}, the conditional assumption that $\Sigma$ has an atom in \cite{Stepanov06}*{Theorem 5.5} holds, and so the theorem follows from \cite{Stepanov06}*{Theorem 5.5}. 
\end{proof}

\subsection{Bounding the number of endpoints of minimizers}\label{sec:branchingrates}

From \cref{theorem:stepanovtopologicalcharacterization}, we know that any minimizer $\Sigma \in \mathcal{S}_{\ell}$ of the average distance problem has only finitely many endpoints points. However, it is not clear how we should expect the number of endpoints points of $\Sigma$ to behave as we vary $\ell$. Intuitively, we may expect fewer endpoints points when $\ell$ is small compared to $\mathrm{diam}(\supp \mu)$; this is supported by the numerical results of \cite{Buttazzo02}*{Figures 3-5 and Appendix B}. It is much less clear how we should expect the number of noncut points to change as $\ell$ grows larger: comparing the numerical results \cite{Buttazzo02}*{Figures 3-5} and \cite{Buttazzo02}*{Figures 10-12} seems to suggest qualitatively different growth rates for the number of endpoints even between the cases when $\mu$ is the uniform measure on the unit ball versus the unit square in $\mathbb{R}^2$.

In this section, we will provide an upper bound on the number of noncut points of a hard-constraint average distance minimizer $\Sigma$ depending on $\ell$ and a quantity defined via the barycentre field. This further deepens our ability to understand the structure of a minimizer $\Sigma$ in terms of its barycentre field. 

Let us remark that we do in general expect there to be some branching present as $\ell$ increases: lower bounds on the time at which a branching must occur in the quasi-static evolution for the average distance functional are provided by Lu \cite{Lu12} for certain configurations.

Let
\begin{equation}\label{eq:noncutSigmadefinition}
    \mathrm{Noncut}(\Sigma) := \{\sigma \in \Sigma \ | \ \sigma \text{ is a noncut point} \}
\end{equation}
be the set of noncut points of $\Sigma$. Define
\[
N_{\Sigma} : = \#\mathrm{Noncut}(\Sigma),
\]
and let 
\begin{equation}\label{eq:branchingratedef}
    \begin{split}
        n(\ell) &:= \min \{N_{\Sigma}\ | \ \Sigma \in \mathcal{S}_{\ell} \text{ is an average distance minimizer}\}; \\
        N(\ell)&:= \max \{N_{\Sigma}\ | \ \Sigma \in \mathcal{S}_{\ell} \text{ is an average distance minimizer}\}.
    \end{split}
\end{equation}
We will refer to $N(\ell)$ and $n(\ell)$ as the \textit{upper and lower endpoint rates}, respectively. This terminology is justified by \cref{theorem:stepanovtopologicalcharacterization} (1), which says that every noncut point of an optimizer $\Sigma$ is an endpoint, and we have that
\[
2 \leq n(\ell) \leq N(\ell) < \infty
\]
for all $\ell$ by \cref{theorem:stepanovtopologicalcharacterization} (2).
\begin{remark}
    It is clear that minimizers of the average distance problem need not be unique: for example, when $\mu$ is the uniform measure on the unit disk, any minimizer remains a minimizer after a rotation. However, it is not clear whether we will have $n(\ell) = N(\ell)$ for every $\ell > 0$.
\end{remark}
We now combine \cref{theorem:boundingmassofnoncutpoints} and \cref{prop:boundingscalingconstantintermsofatoms} to obtain a version of \cite{Buttazzo03}*{Proposition 7.1} in which the constant is computed in terms of the barycentre field.

\begin{lemma}[Comparison of noncut points]\label{lemma:boundingnoncutpointintermsofnoncutpoint}
    Assume $\phi \in C^{1,1}_{\mathrm{loc}}([0, \infty))$ is strictly increasing, and satisfies \eqref{eq:alpha2}. Let $\Sigma \in \mathcal{S}_{\ell}$ be optimal. Then, for any noncut points $\sigma_1, \sigma_2 \in \Sigma$,
    \[
    \frac{1}{4d^{3/2}}\rho_{\pi_{\Sigma}}\{\sigma_1\} \leq |\mathcal{B}_{\pi_{\Sigma}}(\sigma_2)|\nu_{\pi_{\Sigma}}\{\sigma_2\}
    \]
\end{lemma}
\begin{proof}
    Combining \cref{prop:boundingscalingconstantintermsofatoms}, \cref{remark:scalingandlipschitzinequality}, and \cref{theorem:boundingmassofnoncutpoints}, we get
    \begin{align*}
        \frac{\ell}{4d^{3/2}}\rho_{\pi_{\Sigma}}\{\sigma_1\} &\leq \beta_{\Sigma} \\ &\leq \sup_{\xi \in \mathrm{Lip}^*(\Sigma)}\frac{1}{\mathrm{Lip}(\xi)}\int_{\Sigma} \xi(\sigma)\cdot \mathcal{B}_{\pi_{\Sigma}}(\sigma)d\nu_{\pi_{\Sigma}}(\sigma) \\ &\leq \ell|\mathcal{B}_{\pi_{\Sigma}}(\sigma_2)|\nu_{\pi_{\Sigma}}\{\sigma_2\},
    \end{align*}
    so dividing through by $\ell$ yields the desired inequality.
\end{proof}

Now, we will use \cref{lemma:boundingnoncutpointintermsofnoncutpoint} to provide an upper bound on the number of noncut points of an average distance minimizer $\Sigma$.
\begin{proposition}[Upper bound on the number of endpoints]\label{prop:initialboundonbranchingrate}
    Assume $\phi \in C^{1,1}_{\mathrm{loc}}([0, \infty))$ satisfies \eqref{eq:alpha2}. Assume that $\mu(B_{\epsilon}(x)) = o(\epsilon)$ for each $x \in \mathbb{R}^d$. Let $\Sigma \in \mathcal{S}_{\ell}$ be optimal, and define  \begin{equation}\label{eq:globalendpointcurvature}
        \Lambda_{\Sigma} := \inf\{\mathrm{Lip}(f) \ | \ \int_{\mathrm{Noncut}(\Sigma)}|\mathcal{B}_{\pi_{\Sigma}}(\sigma)| d\nu(\sigma) \leq \int_{\Sigma}f(\sigma)\cdot \mathcal{B}_{\pi_{\Sigma}}(\sigma)d\nu(\sigma)\}.
    \end{equation}
    Then, for any $\sigma^* \in \mathrm{Noncut}(\Sigma)$,
    \[
    N_{\Sigma} \leq 4d^{3/2}\Lambda_{\Sigma}\ell.
    \]
\end{proposition}
\begin{proof}
    Let $f$ be a Lipschitz function satisfying the defining inequality in \eqref{eq:globalendpointcurvature}. 
    By \cref{lemma:boundingnoncutpointintermsofnoncutpoint}, for any noncut point $\sigma^* \in \mathrm{Noncut}(\Sigma)$,
    \[
    \frac{1}{4d^{3/2}}\rho_{\pi_{\Sigma}}\{\sigma^*\} \leq |\mathcal{B}_{\pi_{\Sigma}}(\sigma^*)|\nu_{\pi_{\Sigma}}\{\sigma^*\}.
    \]
    Since $\mathrm{Noncut}(\Sigma)$ is finite by \cref{theorem:stepanovtopologicalcharacterization}, we may take the sum over all $\sigma^* \in \mathrm{Noncut}(\Sigma)$ to yield
    \begin{align*}
         \frac{1}{4d^{3/2}}\rho(\mathrm{Noncut}(\Sigma)) &\leq \int_{\mathrm{Noncut}(\Sigma)}|\mathcal{B}_{\pi_{\Sigma}}(\sigma)|d\nu(\sigma)\\
         &\leq \int_{\Sigma}f(\sigma)\cdot \mathcal{B}_{\pi_{\Sigma}}(\sigma)d\nu(\sigma).
    \end{align*}
    But by \cref{theorem:boundingmassofnoncutpoints}, for any noncut point $\sigma^* \in \mathrm{Noncut}(\Sigma)$, we have
    \begin{align*}   \int_{\Sigma}f(\sigma)\cdot \mathcal{B}_{\pi_{\Sigma}}(\sigma)d\nu(\sigma) &\leq \mathrm{Lip}(f) \sup_{\xi \in \mathrm{Lip}^*(\Sigma)}\int_{\Sigma} \xi(\sigma)\cdot \mathcal{B}_{\pi_{\Sigma}}(\sigma)d\nu_{\pi_{\Sigma}}(\sigma) \\
    &\leq \mathrm{Lip}(f)\ell|\mathcal{B}_{\pi_{\Sigma}}(\sigma^*)|\nu_{\pi_{\Sigma}}(\sigma^*),
    \end{align*}
    and so
    \[
    \frac{1}{4d^{3/2}}\rho(\mathrm{Noncut}(\Sigma)) \leq \mathrm{Lip}(f)\ell|\mathcal{B}_{\pi_{\Sigma}}(\sigma^*)|\nu\{\sigma^*\}.
    \]
    Summing over all $\sigma^* \in \mathrm{Noncut}(\Sigma)$ again, we get
    \[
    N_{\Sigma}\frac{1}{4d^{3/2}}\rho(\mathrm{Noncut}(\Sigma)) \leq \mathrm{Lip}(f)\ell\int_{\mathrm{Noncut}(\Sigma)}|\mathcal{B}_{\pi_{\Sigma}}(\sigma)|  d\nu(\sigma).
    \]
    Recall that by the triangle inequality, for any measurable $E$, we have
    \[
    \int_{E}|\mathcal{B}_{\pi_{\Sigma}}(\sigma)| d\nu(\sigma) \leq \rho(E).
    \]
    So, since $\rho(\mathrm{Noncut}(\Sigma)) > 0$ by Corollary \ref{cor:nontrivialbarycentrefieldandatoms} and Theorem \ref{theorem:existenceofatom}, by taking the infimum over all $f$ satisfying the inequality in \eqref{eq:globalendpointcurvature} we conclude that
    \[
    N_{\Sigma} \leq 4d^{3/2}\Lambda_{\Sigma}\ell.
    \]  
\end{proof}

Even in the simplest cases, little is currently known about the behaviour of $N(\ell)$: indeed, it is not even clear whether $N(\ell) \to \infty$ as $\ell \to \infty$ when $\mu$ is the uniform measure on the unit ball in $\mathbb{R}^2$, c.f. \cite{Buttazzo02}*{Figures 3-5}. A better understanding of the behaviour of $n(\ell)$ and $N(\ell)$ would be very interesting, and \cref{prop:initialboundonbranchingrate} exhibits the usefulness of the barycentre field for studying this problem.

\end{document}